\documentclass[a4paper]{article}

\usepackage{amsmath,amsfonts,amssymb,amsthm}
\usepackage{graphicx}
\usepackage{ifthen}
\usepackage{nicefrac}
\usepackage{MnSymbol}
\usepackage{comment}
\usepackage{pgfplots}

\usepackage{accents}
\renewcommand*{\dot}[1]{%
  \accentset{\mbox{\large\bfseries .}}{#1}}

\usepackage{todonotes}
\numberwithin{equation}{section}

\DeclareMathOperator{\sgn}{sgn}
\DeclareMathOperator{\imag}{i}

\newcommand{\vx}{\mathbf{x}}
\newcommand{\vn}{\mathbf{n}}

\newcommand{\supp}{\operatorname{supp}}

\newtheorem{theorem}{Theorem}
\newtheorem{lemma}[theorem]{Lemma}
\newtheorem{remark}[theorem]{Remark}

\begin{document}

\title{On a multiscale formulation for multiperforated plates}

\author{Kersten Schmidt\thanks{Department of Mathematics, Technical University of
Darmstadt, Dolivostr. 15, 64293 Darmstadt, Germany} \and Sven Pfaff$^*$}%

\maketitle

\begin{abstract}
  Multiperforated plates exhibit high gradients and a loss of regularity concentrated in a boundary layer and so a direct numerical simulation becomes very expensive. %
  For elliptic differential equations the solution at some distance of the boundary is only affected in an effective way and the macroscopic and mesoscopic                                                                                                                                                                                                                                                                                                                                                                                                                                                                                                                                                                                                                                                                                                                                                                                                                                                                                                                                                                                                                                                                                                                                                                                                                                                                                                                                                                                                                                                                             behaviour can be separated. %
  A~multiscale formulation in the spirit of the heterogeneous multiscale method
  is introduced on the example of the Poisson equation. %
  Based on the method of matched asymptotic expansion the solution is separated into a macroscopic far field defined in a domain with only slowly varying boundary and
  a mesoscopic near field defined in scaled coordinates on possibly varying infinite periodicity cells. %
  The near field has a polynomial behaviour that is coupled to the traces of the macroscopic variable on the mid-line of the multiperforated plate. %
  A variational formulation using a Beppo-Levi space in the strip is introduced and its well-posedness is shown. The variational framework when truncating the infinite strip is discussed and the truncation error is estimated.
\end{abstract}

%


\pagestyle{myheadings}
\thispagestyle{plain}
\markboth{K. Schmidt and S. Pfaff}{On a multiscale formulation for multiperforated plates}

\section{Introduction}

This paper considers the solution of second order elliptic differential equations in presence of multiperforated plates or thin mesh like structures with locally periodic pattern that may be vary on a macroscopic scale. Multiperforates plates can be applied to reduce acoustic noise pollution~\cite{Lahiri.Sadig.Gerendas.Enghardt.Bake:2011} in lecture halls, concert halls or in car mufflers. They can be applied to suppress thermoacoustic instabilities and for cooling in combustion chambers or as sieves to control fluid flows. %
Moreover, thin mesh-like structures consisting of metallic wires -- so called Faraday cage effect -- can effectively shield  electric fields, and similar structures of elastic rods -- the stents -- are used in blood vessels in human body where they influence the flow.

Due to the multiple scales in the geometry and consequently the solution a direct numerical simulation of such problems is very costly. %
If finite element methods are used the mesh-width needs to be as small as the smallest geometrical scale, at least in an adaptive refinement towards the perforated plate or thin mesh-like structure. %

Therefore, models for the macroscopic fields are proposed that take the thin  microstructures with effective boundary or transmission conditions into account. %
Their derivation relies on the observation that the small scale variations of the solution
have a boundary layer behaviour and decay exponentially away from the thin microstructure.
This boundary layer have been widely studied since the works of Sanchez-Palencia~\cite{SanchezPalencia:1980,SanchezPalencia:1985}, Achdou~\cite{Achdou:1989,Achdou:1992} and Artola and Cessanat~\cite{Artola.Cessenat:1991,Artola.Cessenat:1991a}.

With a combination of homogenisation techniques 
and the method of matched asymptotic expansions 
or multiscale expansions 
an asymptotic expansion of the near field and far field solution can be derived. 
This asymptotic technique is sometimes called \emph{surface homogenisation}. %

An asymptotic expansion of order~1 has been obtained for the acoustic wave propagation through an perforated duct of vanishing thickness~\cite{Bonnet.Drissi.Gmati:2004}, where in the limit of vanishing period the perforated duct becomes transparent~\cite{Bonnet.Drissi.Gmati:2005}. %
For the scattering by a thin ring of regularly-spaced inclusions 
an asymptotic expansion of any order has been derived and justified in~\cite{Claeys.Delourme:2013}, and in~\cite{Delourme.Haddar.Joly:2012} an approximative model with transmission conditions of order~2 has been derived and justified. %
In~\cite{Chapman.Hewett.Trefethen:2015} an approximate model for the Poisson problem with regularly spaced small inclusions with Dirichlet boundary conditions is derived with a three-scale expansion where the size of the inclusions and the distance to their nearest neighbour are considered as independent scales, which is extended to the Helmholtz equation in~\cite{Hewett.Hewitt:2016}. %
Approximate boundary conditions for regularly spaced inclusions of a different material as in the surrounding area has been derived for the Poisson equation via a two-scale homogenisation in~\cite{Marigo.Maurel:2016-2,Maurel.Marigo.Ourir:2016}, and the method was applied for inclusions with Dirichlet or Neumann boundary conditions in~\cite{Delourme.Luneville.Marigo.Maurel.Mercier.Pham:2021}.

Alternatively to the surface homogenisation the periodic unfolding method~\cite{Cionarescu.Damlamiam.Griso:2002} that is based on the two-scale convergence~\cite{Nguetseng:1989} 
was extended to a multiperforated plates for the Helmholtz equation in~\cite{Lukes.Rohan:2007,Cioranescu.Damlamian.Griso.Onofrei:2008}. An approximative model for the acoustic-structure interaction with elastic multiperforated plates was derived with periodic unfolding in~\cite{Rohan.Lukes:2019}.

For multiperforated acoustic liners with small viscosity a third scale for the hole size has been considered in surface homogenisation to obtain approximative models and transmission conditions in two dimensions~\cite{Semin.Schmidt:2018} and three dimensions~\cite{Schmidt.Semin.ThoensZueva.Bake:2018}. %
The surface homogenisation can be extended to multiperforated plates of finite size by incorporating additional terms for corner singularities~\mbox{\cite{Delourme.Schmidt.Semin:2016,Semin.Delourme.Schmidt:2018}}.

Based on the homogenisation theory~\cite{Bensoussan.Lions.Papanicolaou:1978,SanchezPalencia:1980} for periodic microstructures in all space directions numerical methods were proposed. %
The heterogenous multiscale method~\cite{E.Engquist:2003} (HMM) aims to provide a numerical solution to the limit equations where local cell problems are solved on quadrature points of the finite element mesh. This allows for locally periodic microstructures, where the local problems may change slowly. A complete numerical analysis of the method in terms of the macroscopic mesh-width and the mesh-width of the local cell problems was given in~\cite{Abdulle:2005,Abdulle.E.Engquist.VandenEijnden:2012}. Matache and Schwab proposed a two-scale finite element method~\cite{Matache.Schwab:2002-2,Matache:2002} whose computational effort is independent of the small scale, even so it is not based on analytic homogenisation and local cell problems, but on two-scale regularity results.

In this paper we aim to propose and analyse a coupled variational formulation for the far and near field that are present in the surface homogenisation. For the coupling the principles of the method of matched asymptotic expansions are applied. Then the near and far field in the variational formulation can be discretised by finite elements which shall be presented in a forthcoming paper.

The article is structured as follows. In Section~\ref{sec:model problem} the geometry and model problem with the solution decomposition in the macroscopic far field and near field is introduced. 
With matching conditions based on the method of matched asymptotic expansions the formulation of the coupled problem is stated. In Section~\ref{sec:near_field} a variational framework for the near field problem in an infinite strip using a Beppo-Levi space is introduced, its well-posedness is shown and the error introduced by the truncation of the strip is estimated. Finally, in Section~\ref{sec:var_coupled} the well-posedness of the coupled formulation is proven and the truncation error is estimated. 


\section{The geometric setting and the model problem}
\label{sec:model problem}

\subsection{Domain with multiperforated plates}

Let $\Omega \subset \mathbb{R}^2$ be an open, bounded Lipschitz domain. %
In this domain we consider a perforated wall that is defined in a vicinity of a closed $C^1$ curve $\Gamma$, that we call mid-line, with unit normal vector $\vn$. %
The curve $\Gamma$ is parametrized with $\vx_\Gamma: [0,1) \to \Gamma$ where $c_\Gamma \leq |\vx_\Gamma'(x)| \leq C_\Gamma$ with $c_\Gamma, C_\Gamma > 0$. %
The normalized normal vector $\vn$ on $\Gamma$ is given by %
 $(\vx_\Gamma'(x))^\bot, x \in [0,1)$ where $\mathbf{v}^\bot = (v_2, -v_1)^\top$ for any $\mathbf{v} = (v_1, v_2)^\top$ is the vector that is turned clock-wise by~$90^\circ$. %
 Using the parametrization of the vicinity of $\Gamma$
\begin{align}
    \boldsymbol{\phi}: (x,y) \mapsto \vx_\Gamma(x) + y \vn_\Gamma(x)\ .
    \label{eq:phi}
\end{align}
we define the perforated domain as
\begin{align}
   \Omega^\varepsilon = \Omega \setminus \bigcup_{n=1}^{\nicefrac{1}{\varepsilon}} \overline{\boldsymbol{\phi} \left( \varepsilon\widehat{\Omega}_w (\boldsymbol{\phi} ( \varepsilon (n-\tfrac{1}{2}), 0))\right)} \ ,
   \label{eq:Omega_eps}
\end{align}
where $\varepsilon > 0$ is a small parameter corresponding to period in the parameter domain $[0,1)$ with $\nicefrac{1}{\varepsilon} \in \mathbb{N}$. %
Here, ~$\widehat{\Omega}_w(\vx)$, $\vx \in \Gamma$ is the local wall pattern, which is for each $\vx \in \Gamma$ an open Lipschitz domain in $(0,1) \times [-R_0,R_0]$ for some $R_0 > 0$. %
The outer normalized normal vector field on $\partial\widehat{\Omega}_w(\vx)$ is denoted by $\widehat{\vn}(\vx, X, Y) = (\widehat{n}_1(\vx, X, Y), \widehat{n}_2(\vx, X, Y))^\top$. %
For simplicity we assume that the local wall pattern match between the left and right side, i.e.,
\begin{align*}
 \widehat{I}(\vx) := \left\{ Y \in \mathbb{R}, (0,Y)^\top \not\in \partial\widehat{\Omega}_w(\vx) \right\} =
 \left\{ Y \in \mathbb{R}, (1,Y)^\top \not\in \partial\widehat{\Omega}_w(\vx) \right\}.
\end{align*}
Moreover, $\widehat{\Omega}(\vx) = ((0,1) \times \mathbb{R}) \setminus \overline{\widehat{\Omega}_w(\vx)}$ denotes the periodicity cell on $\vx \in \Gamma$. %
Assuming $\widehat{\Omega}_w(\vx)$ to depend continuously on $\vx$ in a finite partition of $\Gamma$ the perforated wall is called \emph{locally periodic}.

\begin{figure}
   \centering
   (a)
   \quad
   \begin{tikzpicture}[scale=1]
     \begin{axis}[enlargelimits=false,xtick=\empty,ytick=\empty]
       \addplot graphics [xmin=0, xmax=1.5, ymin=0, ymax=1.5] {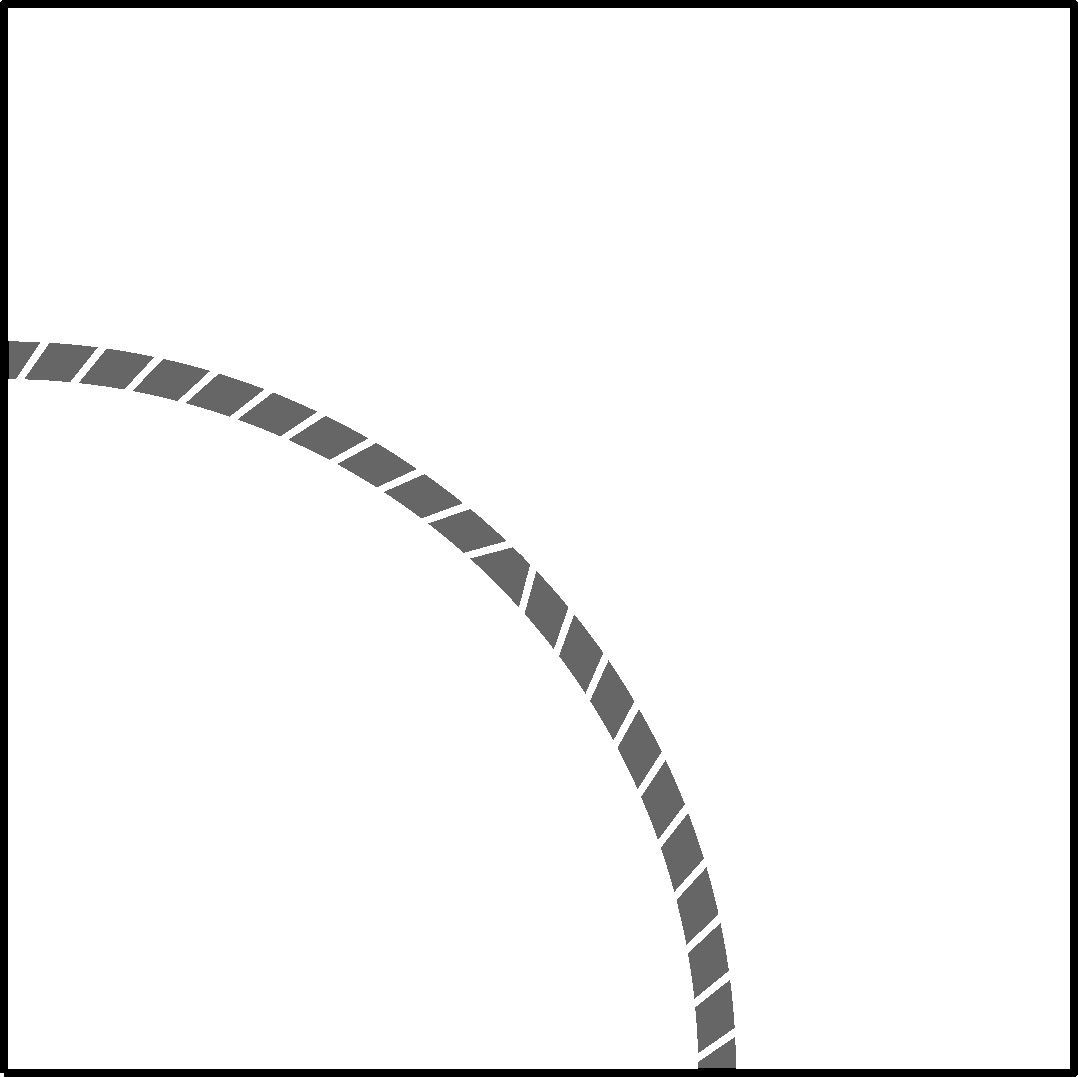};
       \addplot[no markers] (1.2,1.2) node {\Large$\Omega^\varepsilon$};
     \end{axis}
  \end{tikzpicture}
  \qquad
  (b)
  \resizebox{0.12\linewidth}{!}{
  \hspace{-8em}
  \begin{tikzpicture}[scale=1]
    \begin{axis}[
    xmin=-0.2,xmax=1.05, ymin=-2.5, ymax=2.7,
    xtick=\empty,ytick=\empty,
                 axis equal,
                 axis line style={draw=none},
                ]
       \addplot graphics [xmin=0, xmax=1, ymin=-1, ymax=1] {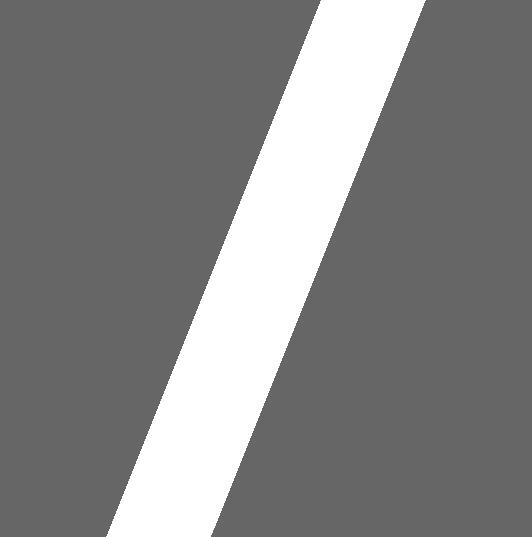};
       \addplot[dashed] coordinates {(0,-10) (0,2.5)};
       \addplot[dashed] coordinates {(1,-10) (1,2.5)};
       \addplot[no markers] ( 0.25, 0.75) node {$\widehat{\Omega}_w$};
    \end{axis}
  \end{tikzpicture}
   \hspace{-8em}
  }
  \resizebox{0.12\linewidth}{!}{
  \hspace{-8em}
  \begin{tikzpicture}[scale=1]
    \begin{axis}[
    xmin=-0.2,xmax=1.05, ymin=-2.5, ymax=2.7,
    xtick=\empty,ytick=\empty,
                 axis equal,
                 axis line style={draw=none},
                 ]
       \addplot graphics [xmin=0, xmax=1, ymin=-1, ymax=1] {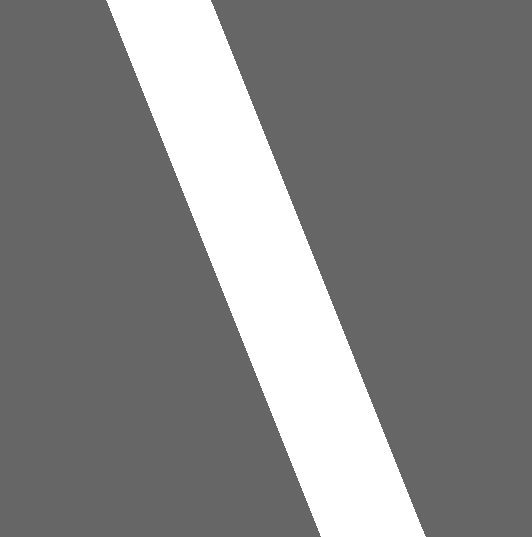};
       \addplot[dashed] coordinates {(0,-2.5) (0,2.5)};
       \addplot[dashed] coordinates {(1,-2.5) (1,2.5)};
       \addplot[no markers] ( 0.75, 0.75) node {$\widehat{\Omega}_w$};
    \end{axis}
  \end{tikzpicture}
  \hspace{-8em}
  }
  %
  \caption{Illustration of (a) a domain $\Omega^\varepsilon$ with perforated wall with (b) two occuring wall pattern $\widehat{\Omega}_w$.}
  \label{fig:Omega_eps}
\end{figure}
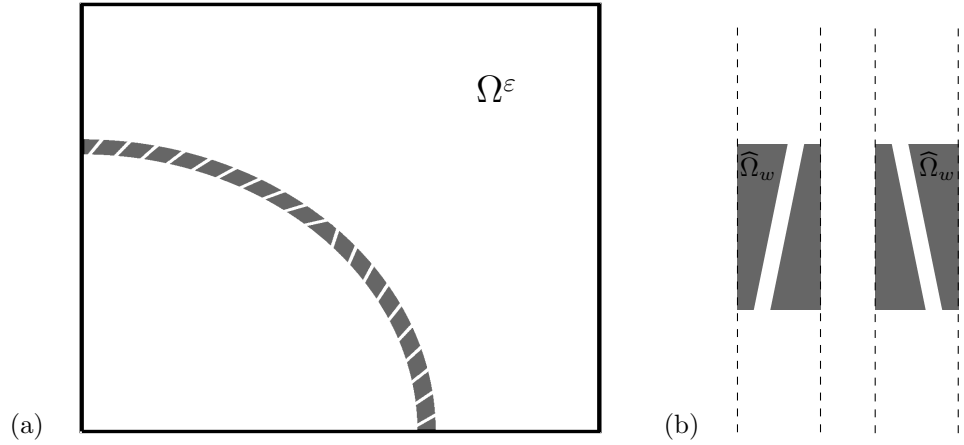

\subsection{Poisson problem in the perforated domain}

In the perforated domain $\Omega^\varepsilon$ we state the model problem
\begin{align}
  \left\{
 \begin{aligned}
    -\Delta u^\varepsilon &= f \text{ in } \Omega^\varepsilon\ ,\\[0.5em]
    \partial_n u^\varepsilon &= 0 \text{ on } \partial\Omega^\varepsilon\ , \\
    \int_{\Omega_f} u^\varepsilon \,\text{d}\vx &= 0\ ,
 \end{aligned}
  \right.
  \label{eq:Poisson:eps}
\end{align}
where $\Omega_f := \supp f \subset \Omega \setminus \Gamma$ with $\operatorname{dist}(\supp f, \Gamma) > 2\sqrt{\varepsilon_0}$ for some $\varepsilon_0 > 0$. %
Assuming further that $\int_\Omega f(\vx) \mathrm{d}\vx = 0$ we can assert that the Poisson problem~\eqref{eq:Poisson:eps} has a unique solution.

Based on the principles of periodic homogenization and the method of matched asymptotic expansions~\cite{Delourme.Haddar.Joly:2012,Claeys.Delourme:2013,Delourme.Schmidt.Semin:2016,Semin.Schmidt:2018} we take the ansatz
\begin{align}
   u^\varepsilon(\vx) &\sim
   \begin{cases}
     U_{\mathrm{int}}\left(\vx_\Gamma(x), \frac{x}{\varepsilon} - \left\lfloor \frac{x}{\varepsilon} \right\rfloor, \frac{y}{\varepsilon}\right)  \text{ with } \vx = \boldsymbol{\phi}(x,y),
     & \operatorname{dist}(\vx, \Gamma) < 2\sqrt{\varepsilon}, \\
     u_{\mathrm{ext}}\left(\vx\right) ,
     & \operatorname{dist}(\vx, \Gamma) > \sqrt{\varepsilon},
   \end{cases}
   \label{eq:ansatz}
\end{align}
where $U_{\mathrm{int}}: \{ (\vx, X, Y), \vx \in \Gamma, (X,Y) \in \widehat{\Omega}(\vx)\} \to \mathbb{R}$ and $u_\mathrm{ext}: \Omega \backslash \Gamma \to \mathbb{R}$ describe the dominating behaviour in a vicinity and outside a vicinity of the microstructure. With ``$\sim$`` we indicate that further terms are smaller.

Since $(\vx_\Gamma(x-\varepsilon), X+1,Y)$ and $(\vx_\Gamma(x), X, Y)$ correspond to the same point $\vx$ in the vcinity of the microstructure and assuming continuity of $u^\varepsilon$ we find that
\begin{align*}
   U_{\mathrm{int}}\left(\vx_\Gamma(x-\varepsilon), \left\lfloor \frac{x}{\varepsilon} \right\rfloor, \frac{y}{\varepsilon}\right)
   \sim U_{\mathrm{int}}\left(\vx_\Gamma(x), \left\lfloor \frac{x}{\varepsilon} \right\rfloor, \frac{y}{\varepsilon}\right)\ .
\end{align*}
As $\varepsilon$ is assumed to be small we consider formally the limit $\varepsilon \to 0$ to justify that $U_\mathrm{int}$ shall be taken as periodic in $X$, i.e., for any $(\vx,Y) \in \Gamma \times \widehat{I}$ it holds
\begin{subequations}
\label{eq:system:Uint}
\begin{align}
   \lim_{X\to 0+} U_{\mathrm{int}}(\vx, X, Y)
     = \lim_{X\to 1-} U_{\mathrm{int}}(\vx, X, Y)\ .
   \label{eq:system:Uint:1}
\end{align}
Similarly, for any $(\vx,Y) \in \Gamma \times \widehat{I}$ it is reasonable to assume
\begin{align}
   \lim_{X\to 0+} \partial_X U_{\mathrm{int}}(\vx, X, Y)
     = \lim_{X\to 1-} \partial_X U_{\mathrm{int}}(\vx, X, Y)\ .
   \label{eq:system:Uint:2}
\end{align}
Inserting the ansatz~\eqref{eq:ansatz}$_1$ into~\eqref{eq:Poisson:eps} we find for $\vx = \boldsymbol{\phi}(x,y)$ that
\begin{align*}
   0 &= -\Delta u^\varepsilon(\vx) \sim -\frac{1}{\varepsilon^2} \left( \frac{\partial^2}{\partial X^2} + \frac{\partial^2}{\partial Y^2} \right) U_{\mathrm{int}}\left(\vx_\Gamma(x), \frac{x}{\varepsilon} - \left\lfloor \frac{x}{\varepsilon} \right\rfloor, \frac{y}{\varepsilon}\right)  \\
   0 &= \partial_n u^\varepsilon(\vx) \sim \frac{1}{\varepsilon} \nabla_{XY} U_{\mathrm{int}}\left(\vx_\Gamma(x), \frac{x}{\varepsilon} - \left\lfloor \frac{x}{\varepsilon} \right\rfloor, \frac{y}{\varepsilon}\right)\cdot \widehat{\vn}(\vx_\Gamma(x)) 
\end{align*}
where ''$\sim$`` indicates that further terms are smaller.
Hence, we demand for all $\vx \in \Gamma$
\begin{align}
   -\Delta_{XY} U_{\mathrm{int}}(\vx, X, Y) &= 0, \quad (X,Y) \in \widehat{\Omega}(\vx),
   \label{eq:system:Uint:3}
   \\
   \nabla_{XY} U_{\mathrm{int}}(\vx, X, Y) \cdot \widehat{\vn}(\vx) &= 0, \quad (X,Y) \in \partial\widehat{\Omega}_w(\vx).
   \label{eq:system:Uint:4}
\end{align}
\end{subequations}
Now, for $|Y| \geq R_0$ we can expand $U_{\mathrm{int}}$ in a Fourier series in $X$,
\begin{align}
   U_{\mathrm{int}}(\vx, X,Y) = \sum_{n=0}^\infty U_{\mathrm{int}, n}^\pm (\vx, Y)
   \exp(2\pi \imag n X), \quad \pm Y \geq R_0,
   \label{eq:Uint:for_Y_large}
\end{align}
and inserting~\eqref{eq:Uint:for_Y_large} into~\eqref{eq:system:Uint} and assuming non-exponential increase in $Y$ leads to
\begin{subequations}
\label{eq:Uint:behaviour_in_Y}
\begin{align}
   \label{eq:Uint:behaviour_in_Y:0}
   U_{\mathrm{int}, 0}^\pm (\vx, Y) &= U_{\mathrm{int}, 0,0}^\pm(\vx) + Y \,U_{\mathrm{int}, 0,1}^\pm(\vx), \\
   \label{eq:Uint:behaviour_in_Y:n}
   U_{\mathrm{int}, n}^\pm (\vx, Y) &= U_{\mathrm{int}, n}^\pm(\vx) \,\exp(-2\pi n |Y|)\ .
\end{align}
\end{subequations}
As for the surface homogenization we find a polynomial behaviour in $Y$ for the constant term in $X$ and an exponential decaying behaviour for all other terms.

Indeed, $U_{\mathrm{int}, 0,1}^+(\vx) = U_{\mathrm{int}, 0,1}^-(\vx)$, i.e., the slopes for $Y \to \infty$ and $Y \to -\infty$ coincide. %
To verify this, we integrate~\eqref{eq:system:Uint:3} over the truncated periodicity cell $\widehat{\Omega}_R = \widehat{\Omega} \cap (0,1) \times (-R,R)$ for some $R > R_0$. Then, applying Gauss's theorem and the periodicity condition~\eqref{eq:system:Uint:2} we obtain
\begin{align*}
   \int_0^1 \partial_Y U_\mathrm{int}(\vx, X, R) \,\text{d}X
   = \int_0^1 \partial_Y U_\mathrm{int}(\vx, X, -R) \,\text{d}X\ .
\end{align*}
Now, taking the limit for $R\to \infty$ using~\eqref{eq:Uint:behaviour_in_Y}, we find that the linear slopes are the same,
\begin{align}
   %
   U_{\mathrm{int},0,1}^+(\vx) =
   \lim_{Y \to +\infty} \partial_Y U_{\mathrm{int}}(\vx, X, Y) =
   \lim_{Y \to -\infty} \partial_Y U_{\mathrm{int}}(\vx, X, Y) =
   U_{\mathrm{int},0,1}^-(\vx) \,.
\end{align}
%
In the following we denote the linear slope and the jump and mean of the constant monomial~by
\begin{align*}
  \alpha(\vx) &:= U_{\mathrm{int}, 0,1}^+(\vx) = U_{\mathrm{int}, 0,1}^-(\vx)\ , \\
  u_\infty(\vx) &:= U_{\mathrm{int}, 0,0}^+(\vx) - U_{\mathrm{int}, 0,0}^-(\vx)\ , \\
  m_\infty(\vx) &:= \tfrac{1}{2} U_{\mathrm{int}, 0,0}^+(\vx) + \tfrac{1}{2} U_{\mathrm{int}, 0,0}^-(\vx)\ .
\end{align*}
The two representations of $u^\varepsilon$ in the ansatz~\eqref{eq:ansatz} shall differ only slightly in the two matching zones where $\sqrt{\varepsilon} < \operatorname{dist}(\vx, \Gamma) < 2 \sqrt{\varepsilon}$, i.e.,
\begin{align}
   u_\mathrm{ext}(\vx) \sim U_{\mathrm{int}}\left(\vx_\Gamma(x), \frac{x}{\varepsilon} - \left\lfloor \frac{x}{\varepsilon} \right\rfloor, \frac{y}{\varepsilon}\right), \quad \vx = \boldsymbol{\phi}(x,y),
   \sqrt{\varepsilon} < y < 2 \sqrt{\varepsilon}\ .
\end{align}
As we assumed the midline~$\Gamma$ to be $C^1$ the Taylor expansion of $u_\mathrm{ext}$
around~$\Gamma$
\begin{align*}
   u_\mathrm{ext}(\vx_\Gamma(x) + y \vn_\Gamma) %
   = u_\mathrm{ext}^\pm(\vx_\Gamma(x)) + y \partial_n u_\mathrm{ext}^\pm(\vx_\Gamma(x)) + o(|y|)
\end{align*}
holds separately for the two sides of $\Gamma$ with
 $\partial_n u_\mathrm{ext}^\pm(\vx_\Gamma(x)) = \lim_{y \to 0\pm} \nabla u_\mathrm{ext}(\vx_\Gamma(x) + y \vn_\Gamma(x))\cdot \vn_\Gamma(x)$. %

In the two matching zones the linear polynomial~\eqref{eq:Uint:behaviour_in_Y:0} is the dominating term of the near field~$U_\mathrm{int}$. Therefore, for all $\vx \in \Gamma$
it holds formally that
\begin{align}
   u_\mathrm{ext}^\pm(\vx) + y \partial_n u_\mathrm{ext}^\pm(\vx)
   \sim U_{\mathrm{int},0,0}^\pm(\vx) + \frac{y}{\varepsilon} \alpha(\vx), \quad \sqrt{\varepsilon} < y < 2 \sqrt{\varepsilon}\ .
\end{align}
Hence,
\begin{align*}
   u_\mathrm{ext}^\pm(\vx) &= U_{\mathrm{int},0,0}^\pm(\vx)\ , &
   \varepsilon\partial_n u_\mathrm{ext}^+(\vx) &= \varepsilon\partial_n u_\mathrm{ext}^-(\vx) = \alpha(\vx)\ .
\end{align*}
For the introduction of the coupled system we define $J: \mathbb{R} \to \mathbb{R}$ as a canonical jump function, cf. Fig.~\ref{Verlauf von J(Y)}, that is a smooth and odd function with $\sgn(Y)J(Y)=\frac12$ for $|Y| > R_1$ for some $R_1 > R_0$,
and vanishing in $[-R_0,R_0]$. Moreover, $[\cdot]$ and $\{\cdot\}$ denote the average and the jump of traces on the mid-line~$\Gamma$.

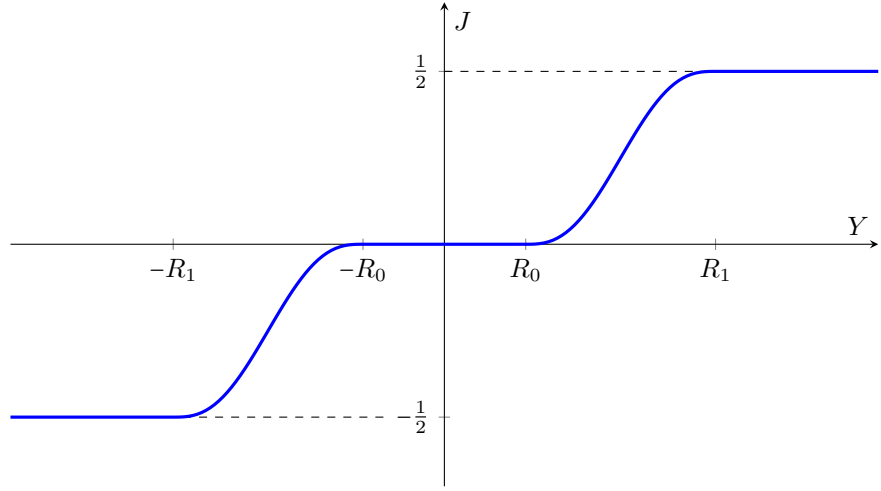
\begin{figure}[ht]
    \centering
    \begin{tikzpicture}
       \begin{axis}
       [axis lines=middle, xlabel=$Y$, ylabel=$J$,
       xmin=-8., xmax=8., ymin=-0.7, ymax = 0.7,
       xtick={-5, -1.5, 0, 1.5, 5},
       xticklabels={$-R_1$, $-R_0$, $0$, $R_0$, $R_1$},
       yticklabels={$-\frac12$, $0$, $\frac12$},
       ytick={-0.5, 0, 0.5},
       width=13cm, height=8cm
       ]
          \addplot[dashed, black] coordinates {( 0,  0.5) ( 6,  0.5) };
          \addplot[dashed, black] coordinates {(-6, -0.5) (-1, -0.5) };
          \addplot[very thick, blue, no markers] table {J.dat};
       \end{axis}
    \end{tikzpicture}

    \caption{The canonical jump function $J(Y)$.}
    \label{Verlauf von J(Y)}
\end{figure}

Altogether, we obtain the coupled system for $(u_\mathrm{ext}, U_\mathrm{int}, \alpha, u_\infty, m_\infty)$
\begin{subequations}
  \label{eq:coupled_system}
    \begin{align}
        \label{eq:coupled_system:1}
        -\Delta u_\text{ext}(\vx) &= f(\vx), \quad  \vx\in\Omega\setminus\Gamma,\\
        \label{eq:coupled_system:2}
        \partial_n u_\text{ext}(\vx) &= 0, \hspace{0.5em} \vx \in \partial\Omega,\\
        \label{eq:coupled_system:3}
        [\partial_n u_\text{ext}](\vx) &= 0, \hspace{0.5em} \vx \in \Gamma,\\
        \label{eq:coupled_system:4}
        -\Delta_{XY}U_\mathrm{int}(\vx,X,Y) &= 0, \hspace{0.5em} \vx\in \Gamma, (X,Y) \in \widehat{\Omega}(\vx),\\
        \label{eq:coupled_system:5}
        \partial_{n} U_\mathrm{int}(\vx,X,Y) &= 0, \hspace{0.5em} \vx \in \Gamma, (X,Y) \in \partial\widehat{\Omega}_w(\vx), \\
        \label{eq:coupled_system:6}
        \lim_{|Y|\to \infty} U_\mathrm{int}(\vx,X,Y) - m_\infty(\vx) - u_\infty(\vx) J(Y) - \alpha(\vx) Y &= 0, \hspace{0.5em} \vx \in \Gamma, X \in (0,1), \\[-0.7em]
        \label{eq:coupled_system:7}
        [u_\text{ext}](\vx) - u_\infty(\vx) &= 0, \hspace{0.5em} \vx \in \Gamma,\\
        \label{eq:coupled_system:8}
        \{u_\text{ext}\}(\vx) - m_\infty(\vx) &= 0, \hspace{0.5em} \vx \in \Gamma,\\
        \label{eq:coupled_system:9}
        \{\partial_n u_\text{ext}\}(\vx) - \varepsilon^{-1} \alpha(\vx)
        &= 0, \hspace{0.5em} \vx \in \Gamma\ .
    \end{align}
\end{subequations}
Here, we use $\partial_{n} U_\mathrm{int}(\vx,X,Y) := \nabla_{XY} U_\mathrm{int}(\vx,X,Y)\cdot \widehat{\vn}(\vx,X,Y)$. The equations~\eqref{eq:coupled_system:1}--\eqref{eq:coupled_system:3} form the subsystem for the macroscopic far field, the equations~\eqref{eq:coupled_system:4}--\eqref{eq:coupled_system:6} form the subsystem for the microscopic near field and~\eqref{eq:coupled_system:7}--\eqref{eq:coupled_system:9} are the coupling conditions.

Note, that the system depends on the small parameter $\varepsilon$ and different choices
of $\varepsilon$ may be considered where also the periodicity cell $\widehat{\Omega}(\mathbf{x})$ may be changed with $\varepsilon$. In this way $(u_\mathrm{ext}, U_\mathrm{int})$ can be considered as an approximation to a particular asymptotic limit for $\varepsilon \to 0$.

In Sec.~\ref{sec:near_field} we discuss the near field problem for given jump $u_\infty$ at infinity that we call near field Dirichlet problem. Then, in Sec.~\ref{sec:var_coupled} we introduce the variational formulation coupling the near and far field and show its well-posedness.

\section{Near field Dirichlet problem in one periodicity cell}
\label{sec:near_field}

In this section we consider a near field problem in one periodicity cell $\widehat{\Omega} = ((0,1) \times \mathbb{R}) \setminus \overline{\widehat{\Omega}_w}$ with $\widehat{\Omega}_w$ denoting a wall domain. Here, we omit the slow variable $\vx$. The solution $U$ satisfies the system
\begin{subequations}
\begin{align}\label{NahfeldInnen}
    -\Delta_{XY} U(X,Y) &= 0 \quad \text{in} \quad \widehat{\Omega}, \\
    \partial_n U(X,Y) &= 0 \quad \text{on} \quad \partial\widehat{\Omega}_w, \\
    U(1,Y) &= U(0,Y) \quad \text{on} \quad \mathbb{R}\setminus I_y, \\
    \partial_X U(1,Y) &= \partial_X U(0,Y) \hspace{0.2cm} \text{on} \hspace{0.2cm} \mathbb{R}\setminus I_y,\\
   \intertext{
      with the polynomial behaviour
   }
 U(X,Y) &= m_\infty + u_{\infty}J(Y) + \alpha Y + O(\exp(-2\pi|Y|)) \quad \text{for} \quad Y \to \pm \infty,
   \label{eq:NahfeldInnen:decay}
\end{align}
\end{subequations}
at infinity, where $m_\infty, u_\infty, \alpha \in \mathbb{R}$.


We consider the problem that we seek the coefficient $\alpha$ for given coefficient $u_{\infty}$, which we call the \emph{near field Dirichlet problem}. As \emph{near field Neumann problem} we would denote the problem where the linear slope $\alpha$ is given and the coefficient $u_\infty$ results. In both problems the solution is defined up to the constant $m_\infty$.

\begin{figure}
   \centering
   \resizebox{0.4\textwidth}{!}{
     \begin{tikzpicture}[scale=1]
        \begin{axis}[enlargelimits=false,xtick=\empty,ytick=\empty]
           \addplot graphics [xmin=0, xmax=1.5, ymin=0, ymax=1.5] {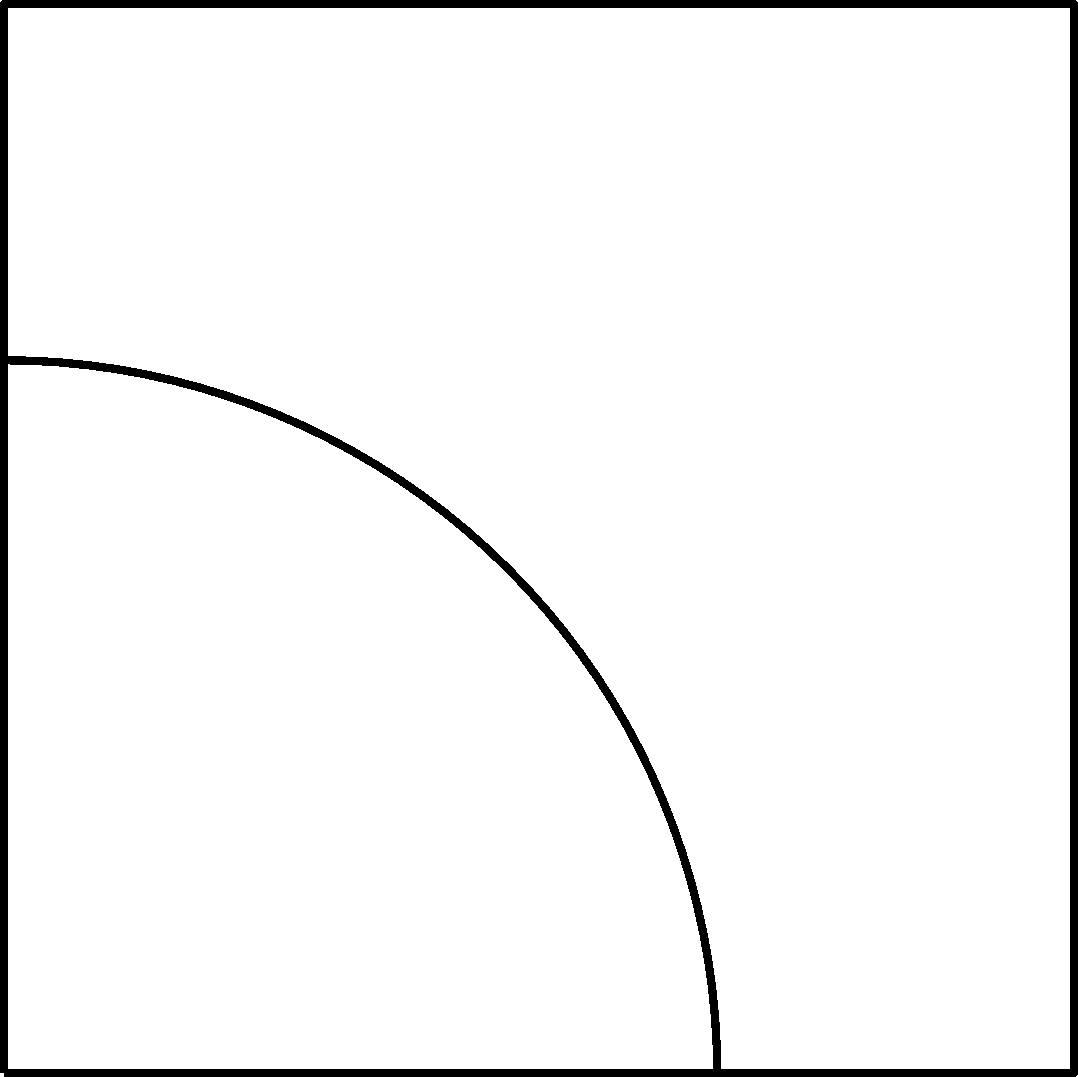};
           \addplot[no markers] (1.2,1.2) node {\Large$\Omega$};
           \addplot[no markers] (1,0.4) node {\Large$\Gamma$};
        \end{axis}
     \end{tikzpicture}}
  \caption{Illustration of the limit domain $\Omega \setminus \Gamma$ with the mid-line $\Gamma$ of the perforated wall.}
\end{figure}     

\begin{figure}
   \centering
     (a)
     \hspace{-1em}
   {\resizebox{0.2\linewidth}{!}{
  \hspace{-7em}
  \begin{tikzpicture}[scale=1]
    \begin{axis}[
    xmin=-1,xmax=1.5, ymin=-2.5, ymax=2.7,
    xtick=\empty,ytick=\empty,
                 axis equal,
                 axis line style={draw=none},
                ]
       \addplot graphics [xmin=0, xmax=1, ymin=-1, ymax=1] {Omega_hat_1.png};
       \addplot[->] coordinates {(-0.5,0) (1.5,0)};
       \addplot[->] coordinates {(-0.3,-2.5) (-0.3,2.5)};
       \addplot[no markers] ( 1.5, 0.2) node {$X$};
       \addplot[no markers] (-0.5, 2.5) node {$Y$};
       \addplot[dashed] coordinates {(0,-2.5) (0,2.5)};
       \addplot[dashed] coordinates {(1,-2.5) (1,2.5)};
       \addplot[no markers] ( 0.5,  1.6 ) node {$\widehat{\Omega}$};
       \addplot[no markers] ( 0.15,-0.2 ) node {$0$};
       \addplot[no markers] ( 1.15,-0.2 ) node {$1$};
       \addplot[no markers] (-0.4 , 0.15) node {$0$};          
       \addplot[no markers] ( 0.27, 0.75) node {$\widehat{\Omega}_w$};
    \end{axis}
  \end{tikzpicture}
  \hspace{-6em}
  }}
  \quad
  (b)
  \hspace{1em}
  {\resizebox{0.2\linewidth}{!}{
  \hspace{-7em}
  \begin{tikzpicture}[scale=1]
    \begin{axis}[
    xmin=-1,xmax=1.5, ymin=-2.5, ymax=2.7,
    xtick=\empty,ytick=\empty,
                 axis equal,
                 axis line style={draw=none},
                ]
       \addplot graphics [xmin=0, xmax=1, ymin=-1, ymax=1] {Omega_hat_1.png};
       \addplot[->] coordinates {(-0.5,0) (1.5,0)};
       \addplot[->] coordinates {(-0.3,-2.5) (-0.3,2.5)};
       \addplot[no markers] ( 1.5, 0.2) node {$X$};
       \addplot[no markers] (-0.5, 2.5) node {$Y$};
       \addplot[dashed] coordinates {(0,-2.0) (0,2.0)};
       \addplot[dashed] coordinates {(1,-2.0) (1,2.0)};
       \addplot[no markers] ( 0.5,  1.6 ) node {$\widehat{\Omega}_R$};
       \addplot[no markers] ( 0.15,-0.2 ) node {$0$};
       \addplot[no markers] ( 1.15,-0.2 ) node {$1$};
       \addplot[no markers] (-0.4 , 0.15) node {$0$};       
       \addplot[no markers] ( 0.27, 0.75) node {$\widehat{\Omega}_w$};
       \addplot[thin] coordinates {(0, 2.0) (1, 2.0)};
       \addplot[thin] coordinates {(0,-2.0) (1,-2.0)};              
       \addplot[thin] coordinates {(-0.4,-2.0) (-0.3,-2.0)};
       \addplot[thin] coordinates {(-0.4, 2.0) (-0.3, 2.0)};
       \addplot[no markers] (-0.4 , 2.0) node[left] {$R$};
       \addplot[no markers] (-0.4 ,-2.0) node[left] {$-R$}; 
    \end{axis}
  \end{tikzpicture}\hspace{-5.35em}}}
  \caption{Illustration of the 
  periodicity cell $\widehat{\Omega}$ and the truncated the periodicity cell $\widehat{\Omega}_{R}$ with $R > R_0$.}
  \label{teilgebiet}
\end{figure}
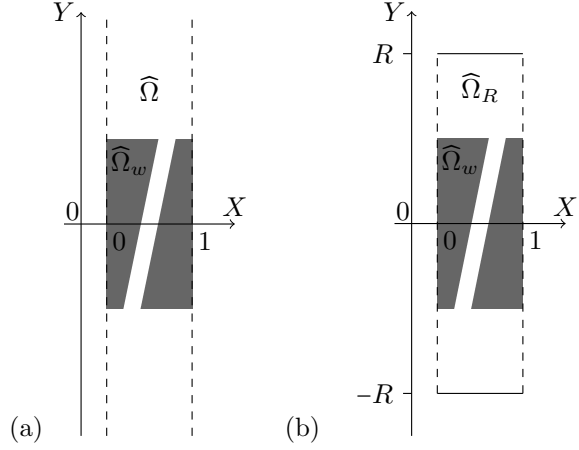

\subsection{Variational spaces}

For the variational problem of the near field Dirichlet problem we consider as unknowns the pair $(\wideparen{U}, \alpha)$ with
\begin{equation}
    \wideparen{U}(X,Y) := U(X,Y) - m_\infty - u_{\infty}J(Y) - \alpha Y,
\end{equation}
that is seeked in the Beppo-Levi space $BL_{0,\sharp}(\widehat{\Omega})$ that is the completion of the space of smooth functions with bounded support and periodicity conditions in $X$
\begin{align*}
   C^\infty_{c,\sharp}(\widehat{\Omega}) := \{ \wideparen{V} \in C^\infty_\sharp(\widehat{\Omega}), \operatorname{diam}(\supp(\wideparen{V})) < \infty \}
\end{align*}
in
\begin{align*}
   BL_\sharp(\widehat{\Omega}) := \left\{
     \wideparen{V} \in \mathcal{D}'_{\sharp}(\widehat{\Omega}) :
     \frac{\wideparen{V}}{\sqrt{1 + |Y|^2}} \in L^2(\widehat{\Omega}) \text{ and }
     \nabla \wideparen{V} \in L^2(\widehat{\Omega})
   \right\}\ .
\end{align*}
Now, we define for any subdomain $\widehat{G} \subseteq \widehat{\Omega}$ the norm
\begin{align}
  \| \wideparen{V} \|_{BL(\widehat{G})}^2
  := \int_{\widehat{G}}
  |\nabla \wideparen{V}(X,Y)|^2 +
   \frac{|\wideparen{V}(X,Y)|^2}{1 + |Y|^2}
  \text{d}(X,Y).
  \label{eq:BL:norm}
\end{align}
The Beppo-Levi spaces $BL_\sharp(\widehat{\Omega})$ and $BL_{0,\sharp}(\widehat{\Omega})$ are Hilbert spaces when endowed with the norm $\| \cdot \|_{BL(\widehat{\Omega})}$
which follows in analogy to~\cite[Chap. XI, Part B]{Dautray.Lions:1990}, see also~\cite{Fratta.Fiorenza:2022}.

For the analysis of the variational formulation in Sec.~\ref{sec:near_field:var_problem} we need the following statement.

\begin{lemma}
  \label{lem:BL}
   The seminorm $|\cdot|_{H^1(\widehat{\Omega})}$ is a norm on the space $BL_{0,\sharp}(\widehat{\Omega})$, which is equivalent to $\|\cdot\|_{BL(\widehat{\Omega})}$.
\end{lemma}
\begin{proof}
   The statement follows from the Poincar\'e inequality
   \begin{align*}
      \int_{\widehat{\Omega}} \frac{|\widehat{V}(X,Y)|^2}{1 + |Y|^2} d(X,Y) \leq C_p
      \int_{\widehat{\Omega}}
  |\nabla \wideparen{V}(X,Y)|^2
  \text{d}(X,Y),
   \end{align*}
   which follows similarly to~\cite[Chap. XI, Part B, Theorem~1]{Dautray.Lions:1990}.
\end{proof}

To see that the functions in the Beppo-Levi--space $BL_{0,\sharp}(\widehat{\Omega})$ decay to zero for $|Y| \to \infty$ we show the statement for the homogeneous Sobolev space $\dot{H}^1_{\sharp}(\widehat{\Omega})$ which is identical to $BL_{0,\sharp}(\widehat{\Omega})$.

\begin{lemma}
   \label{lem:dotH1sharp}
   The space $\dot{H}^1_{\sharp}(\widehat{\Omega})$ as the completion of $C^\infty_c(\widehat{\Omega})$ with periodicity condition in $X$ with respect to the $H^1(\widehat{\Omega})$-seminorm is an Hilbert space with
   \begin{align}
     \int_0^1 \wideparen{V}^2(X,Y) \,\text{d}X \to 0 \quad \text{for } |Y|\to \infty
     \label{eq:dotH1sharp:1}
   \end{align}
   for all $\wideparen{V} \in \dot{H}^1_{\sharp}(\widehat{\Omega})$, 
   that is identical to $BL_{0,\sharp}(\widehat{\Omega})$.
\end{lemma}
\begin{proof}
  First, we show the decaying property~\eqref{eq:dotH1sharp:1}. For this, let for an arbitrary $\wideparen{V} \in \dot{H}^1_{\sharp}(\widehat{\Omega})$ and any $Y$ with $|Y| \geq R_0$ be
  \begin{align*}
     \wideparen{\wideparen{V}}(Y) := \sqrt{\int_0^1 \wideparen{V}^2(X,Y) \,\text{d}X}\ .
  \end{align*}
  By the Cauchy-Schwarz inequality we can assert that
     $\left| \wideparen{\wideparen{V}}'(Y) \right|
     \leq \sqrt{\int_0^1 (\partial_Y \wideparen{V}(X,Y))^2 \,\text{d}X}$
  and so
  \begin{align*}
     \left| \wideparen{\wideparen{V}} \right|_{H^1((-\infty, -R_0) \cup (R_0, \infty))} 
     \leq \left|\wideparen{V} \right|_{H^1(\widehat{\Omega})}
  \end{align*}
  Hence, the continous extension $\wideparen{\wideparen{V}}$ onto $\mathbb{R}$ by linear polynomial into $[-R_0, R_0]$ is the homogeneous Sobolev space $\dot{H}^1(\mathbb{R})$~\cite{EvansBook}, the completion of $C^\infty_c(\mathbb{R})$ with respect to the $H^1(\mathbb{R})$ seminorm. It is well-known that the homogeneous Sobolev space $\dot{H}^1(\mathbb{R})$ with a decaying behavior towards $\pm \infty$.
  This implies~\eqref{eq:dotH1sharp:1}.

  It suffices to show definitness of the $H^1(\widehat{\Omega})$-seminorm on  
  $\dot{H}^1_{\sharp}(\widehat{\Omega})$. For this let $\wideparen{V} \in \dot{H}^1_{\sharp}(\widehat{\Omega})$ be a function with $|\wideparen{V}|_{H^1(\widehat{\Omega})} = 0$. Then, $\nabla \wideparen{V} = 0$ and $\wideparen{V}$ is a constant. Finally, ~\eqref{eq:dotH1sharp:1} implies $\wideparen{V} = 0$, and the definiteness of the $H^1(\widehat{\Omega})$-seminorm follows.
  
    Since the $H^1(\widehat{\Omega})$-seminorm is a norm on the two complete spaces $BL_{0,\sharp}(\widehat{\Omega})$ and $\dot{H}^1_{\sharp}(\widehat{\Omega})$ they are identical which finishes the proof.
\end{proof}

\subsection{Variational problem}
\label{sec:near_field:var_problem}

The unknown $\wideparen{U}$ satisfies
\begin{subequations}
  \label{U:paren}
\begin{align}
    -\Delta \wideparen{U}(X,Y) &= -\Delta U(X,Y) + \alpha \Delta Y + u_{\infty}J''(Y) =  u_{\infty}J''(Y) \quad \text{in} \quad \widehat{\Omega},\label{U:paren:Delta}\\
    \partial_n \wideparen{U}(X,Y) &= \partial_n U(X,Y) -\alpha\partial_n Y - u_{\infty}J'(Y)\widehat{n}_2
    = -\alpha \widehat{n}_2\quad \hspace{1.4em}\text{on} \quad \partial \widehat{\Omega},
\end{align}
\end{subequations}
since $J'(Y) \widehat{n}_2 = 0$ on $\partial\widehat{\Omega}$.
Multiplying \eqref{U:paren:Delta} with a test function $\wideparen{V} \in BL_{0,\sharp}(\widehat{\Omega})$, integrating over $\widehat{\Omega}$ and using integration by parts we find the equality
\begin{multline}
     \int_{\widehat{\Omega}}\nabla \wideparen{U}(X,Y) \cdot \nabla \wideparen{V}(X,Y) \text{d}(X,Y) +  \alpha  \int_{\partial\widehat{\Omega}}  \widehat{n}_2 \wideparen{V}(X,Y) \text{d}_{XY}\sigma 
    = u_{\infty} \int_{\widehat{\Omega}} J''(Y)\wideparen{V}(X,Y) \text{d}(X,Y).
\end{multline}
To derive a second equation for $\wideparen{U}$ we first consider the truncated periodicity cell $\widehat{\Omega}_R := \widehat{\Omega} \cap (0,1) \times (-R,R)$ for $R > R_0$. Applying Green's formula twice
and using $\Delta Y = 0$ we can assert that
\begin{align*}
    \int_{\partial\widehat{\Omega}} &\wideparen{U}(X,Y) \widehat{n}_2 \text{d}_{XY}\sigma =
    \int_{\partial\widehat{\Omega}_R} \wideparen{U}(X,Y)\nabla Y \cdot \mathbf{n} \, \text{d}_{XY}\sigma
    - \int_0^1 \wideparen{U}(X,R) - \wideparen{U}(X,-R) \text{d}X
    \\
    &= \int_{\widehat{\Omega}_R} \nabla \wideparen{U}(X,Y) \cdot \nabla Y  \text{d}(X,Y)
    - \int_0^1 \wideparen{U}(X,R) - \wideparen{U}(X,-R) \text{d}X \\
    &=  - \int_{\widehat{\Omega}_R} Y\Delta \wideparen{U}(X,Y) \text{d}(X,Y)
        + \int_{\partial\widehat{\Omega}} Y \partial_n \wideparen{U}(X,Y) \text{d}_{XY}\sigma \\
    &\hspace{1em}
    + \int_0^1 R\partial_Y \wideparen{U}(X,R) \text{d}X - R\partial_Y \wideparen{U}(X,-R) \text{d}X
    - \int_0^1 \wideparen{U}(X,R) - \wideparen{U}(X,-R) \text{d}X\ .
\end{align*}
Taking the limit for $R \to \infty$ the last two integrals vanish due to the exponential decay of $\wideparen{U}$, see~\eqref{eq:NahfeldInnen:decay}, and inserting~\eqref{U:paren} we find
\begin{align*}
    \int_{\partial\widehat{\Omega}} \wideparen{U}(X,Y) \widehat{n}_2 \text{d}_{XY}\sigma &=
    u_{\infty}\int_{\widehat{\Omega}} Y J''(Y)  \text{d}(X,Y) - \alpha\int_{\partial\widehat{\Omega}} Y  \widehat{n}_2 \text{d}_{XY}\sigma.
\end{align*}
Integrating the first integral on the right hand side by parts twice and using the smoothness of $J$ we can assert that
\begin{align*}
   \int_{\widehat{\Omega}} Y J''(Y) \, \text{d}(X,Y) = -1\ .
\end{align*}
Moreover, the last integral can be simplified using the Green's formula inside the wall~$\widehat{\Omega}_w$,
\begin{align*}
    &\int_{\partial\widehat{\Omega}} Y  \widehat{n}_2 \text{d}_{XY}\sigma = \int_{\partial\widehat{\Omega}} Y  \nabla Y \cdot \mathbf{n} \text{d}_{XY}\sigma
    = -\int_{\widehat{\Omega}_w} Y  \Delta Y + \nabla Y \cdot\nabla Y \text{d}(X,Y) = -\vert\widehat{\Omega}_w\vert < 0,
\end{align*}
where the sign is changed as the normal vector $\mathbf{n}$ is directed inside $\widehat{\Omega}_w$.

Hence, we seek $(\wideparen{U}, \alpha) \in BL_{0,\sharp}(\widehat{\Omega}) \times \mathbb{R}$ such that
\begin{subequations}
\label{eq:var:Uparen_alpha}
\begin{align}
    \int_{\widehat{\Omega}}\nabla \wideparen{U}\cdot\nabla \wideparen{V} \text{d}(X,Y) +  \alpha  \int_{\partial\widehat{\Omega}}  \widehat{n}_2 \wideparen{V} \text{d}_{XY}\sigma
    &= u_{\infty} \int_{\widehat{\Omega}} J'' \wideparen{V} \text{d}(X,Y) \quad \forall \wideparen{V}\in BL_{0,\sharp}(\widehat{\Omega}), \label{eq:var:Uparen_alpha:1}\\
   -\int_{\partial \widehat{\Omega}} \wideparen{U} \widehat{n}_2 \text{d}_{XY}\sigma +\alpha \vert \widehat{\Omega}_w \vert
   &= u_{\infty}\ .
   \label{eq:var:Uparen_alpha:2}
\end{align}
\end{subequations}

\begin{lemma}
  \label{lem:var:Uparen_alpha}
   The variational formulation~\eqref{eq:var:Uparen_alpha} admits a unique solution $(\wideparen{U},\alpha) \in BL_{0,\sharp}(\widehat{\Omega}) \times \mathbb{R}$ and there exists a constant $C$ such that
 \begin{equation}
     \vert \wideparen{U} \vert_{H^1(\widehat{\Omega})} + \vert \alpha \vert \leq C \vert u_{\infty} \vert .\label{eq:stability:Uparen_alpha}
 \end{equation}
\end{lemma}
\begin{proof}
With the bilinear form $\mathsf{a}$ given by
\begin{align*}
   \mathsf{a}((\wideparen{U},\alpha),(\wideparen{V},\beta)) &= \int_{\widehat{\Omega}}\nabla \wideparen{U}\cdot\nabla \wideparen{V} \text{d}(X,Y) +  \alpha  \int_{\partial\widehat{\Omega}}  \widehat{n}_2 \wideparen{V} \text{d}_{XY}\sigma -\beta \int_{\partial \widehat{\Omega}} \wideparen{U} \widehat{n}_2 \text{d}_{XY}\sigma +\alpha\beta\vert \widehat{\Omega}_w \vert
\end{align*}
and the linear form $\ell$ defined by
\begin{align*}
   \ell((\wideparen{V}, \beta)) &= u_{\infty} \left( \int_{\widehat{\Omega}} \wideparen{V} J''  \text{d}(X,Y) + \beta \right),
\end{align*}
that are both continuous on $(BL_{0,\sharp}(\widehat{\Omega}),\mathbb{R})$, the variational formulation~\eqref{eq:var:Uparen_alpha} is equivalent to seek $(\wideparen{U},\alpha) \in BL_{0,\sharp}(\widehat{\Omega}) \times \mathbb{R}$ such that
\begin{align*}
  \mathsf{a}((\wideparen{U}, \alpha), (\wideparen{V},\beta)) = \ell((\wideparen{V},\beta) \quad \text{for all } (\wideparen{V},\beta) \in BL_{0,\sharp}(\widehat{\Omega}) \times \mathbb{R}.
\end{align*}
Taking $(\wideparen{V},\beta) = (\wideparen{U},\alpha)$ we can assert that
\begin{align*}
    a((\wideparen{U},\alpha),(\wideparen{U},\alpha)) = \int_{\widehat{\Omega}}\vert\nabla \wideparen{U}\vert^2 \text{d}(X,Y)  +\alpha^2\vert \widehat{\Omega}_w \vert \geq \min\left(1,\vert\widehat{\Omega}_w \vert\right) \left(\vert \wideparen{U} \vert_{H^1(\widehat{\Omega})} + \vert\alpha\vert^2\right)
\end{align*}
and as $H^1(\widehat{\Omega})$-seminorm is a norm on $BL_{0,\sharp}(\widehat{\Omega})$
due to Lemma~\ref{lem:BL}
the bilinear form is $(BL_{0,\sharp}(\widehat{\Omega}),\mathbb{R})$-elliptic with ellipticity constant $\gamma =  \min(1,|\widehat{\Omega}_w\vert) $.
Hence, the Lax-Milgram lemma implies existence and uniqueness of the solution as well as its continuous dependency on $u_\infty$.
\end{proof}

\subsection{Formulation on the truncated periodicity cell}
\label{AbschneidefehlerInnen}

To be able to propose a numerical scheme we aim to truncate the periodicity cell at $|Y| = R$ for some $R > R_1$, where we search for approximate solutions with homogeneous Dirichlet boundary conditions at $|Y| = R$. %
Imposing this boundary condition is reasonable due to the exponential decay of the solution $\wideparen{U}$ of \ref{eq:var:Uparen_alpha} for $|Y| \to \infty$. %
We denote the truncated periodicity cell $\widehat{\Omega}_R = \widehat{\Omega} \cap  [0,1] \times [-R,R]$ on which we consider the Hilbert space
\begin{equation}
    BL_{0, \sharp}(\widehat{\Omega}_R) :=
    \left\{ \wideparen{V}_R \in \mathcal{D}'_\sharp(\widehat{\Omega}_R) :
    \left\|\wideparen{V}_R \right\|_{BL(\widehat{\Omega}_R)} < \infty,
    \wideparen{V}_R(\cdot,\pm R)=0\right\},
\end{equation} 
By the definiton of $BL_{0, \sharp}(\widehat{\Omega})$ as the completion of $C^\infty_{c,\sharp}(\widehat{\Omega})$
the union of $BL_{0, \sharp}(\widehat{\Omega}_R)$ for all $R > R_0$ each extended by zero for $|Y| > R$ is dense in $BL_{0, \sharp}(\widehat{\Omega})$.
Hence, in view of Lemma~\ref{lem:BL} it follows that the $H^1(\widehat{\Omega}_R)$-seminorm and $BL(\widehat{\Omega}_R)$-norm are equivalent with constants independent of $R$.

Then, we search $(\wideparen{U}_{R},\alpha_R) \in BL_{0, \sharp}(\widehat{\Omega}_R) \times \mathbb{R}$ such that for all $\wideparen{V}_R\in BL_{0, \sharp}(\widehat{\Omega}_R) $
\begin{subequations}
   \label{eq:var:UparenR_alphaR}
\begin{align}
    \int_{\widehat{\Omega}_R}\nabla \wideparen{U}_{R}\cdot \nabla \wideparen{V}_R \,\text{d}(X,Y) +  \alpha_R  \int_{\partial\widehat{\Omega}_R}  \widehat{n}_2 \wideparen{V}_R \text{d}_{XY}\sigma &= u_{\infty} \int_{\widehat{\Omega}_R} J'' \wideparen{V}_R \text{d}(X,Y)
    \\
    -\int_{\partial \widehat{\Omega}_R} \wideparen{U}_{R} \widehat{n}_2 \text{d}_{XY}\sigma +\alpha_R\vert \widehat{\Omega}_w \vert
    &= u_{\infty} \ .
\end{align}
\end{subequations}
\begin{lemma}
\label{lem:var:UparenR_alphaR}
The variational formulation~\eqref{eq:var:UparenR_alphaR} admits a unique solution $(\wideparen{U}_R,\alpha_R) \in BL_{0, \sharp}(\widehat{\Omega}_R) \times \mathbb{R}$ and there exists a constant $C$ independent of $R$ such that
\begin{equation}
    \vert \wideparen{U}_{R} \vert_{H^1(\widehat{\Omega}_R)} + \vert \alpha_R \vert \leq C \vert u_{\infty} \vert.
    \label{eq:stability:UparenR_alphaR}
\end{equation}  
\end{lemma}
\begin{proof}
The proof is in analogy to the one of Lemma~\ref{lem:var:Uparen_alpha} where the bilinear form is $BL_{0, \sharp}(\widehat{\Omega}_R)  \times \mathbb{R}$-elliptic in this case.
\end{proof}

The truncation causes as error that decays exponentially with $R$ as stated in

\begin{lemma}
   \label{lem:UparenR_alpha:err}
   Let $R > 2R_1$. Then, there exists a constant $C$ independent of $R$ and 
   \begin{align*}
      \vert \wideparen{U}_{R}- \wideparen{U}\vert_{H^1(\widehat{\Omega})} + \vert \alpha_R -\alpha \vert \leq C \exp(-\pi R)\ .
   \end{align*}
\end{lemma}
\begin{proof}
  The proof is divided into two parts. First the truncation error is bounded using Céa's lemma by the best approximation error which is then bounded by the error of an interpolant.

  As $BL_{0, \sharp}(\widehat{\Omega}_R)$
  extended by $0$ into $\widehat{\Omega} \setminus \widehat{\Omega}_R$ is a subspace of $BL_{0, \sharp}(\widehat{\Omega})$ and with the $BL_{0, \sharp}(\widehat{\Omega}) \times \mathbb{R}$-ellipticity of the bilinear form $\mathsf{a}$
  we can apply Céa's lemma. This leads to
  \begin{align*}
     \vert \wideparen{U}_{R}- \wideparen{U}\vert_{H^1(\widehat{\Omega})} + \vert \alpha_R -\alpha \vert \leq
     \left(1+\frac{C^2}{\gamma^2}\right)\inf_{(\wideparen{V}_R,\beta_R)\in H^1_{\pm R,\sharp}(\widehat{\Omega}) \times \mathbb{R}}\left(\vert \wideparen{V}_R- \wideparen{U}\vert_{H^1(\widehat{\Omega})}^2 + \vert \beta_R -\alpha \vert^2\right),
  \end{align*}
  where $\gamma = \min(1, |\widehat{\Omega}_w|)$ is the ellipicity constant and $C$ the continuity constant of the bilinear form. As $\beta_R$ can be chosing to equal $\alpha$ as they are real numbers this simplifies to
  \begin{align}
     \vert \wideparen{U}_{R}- \wideparen{U}\vert_{H^1(\widehat{\Omega})} + \vert \alpha_R -\alpha \vert \leq
     \left(1+\frac{C^2}{\gamma^2}\right)\inf_{\wideparen{V}_R\in H^1_{\pm R,\sharp}(\widehat{\Omega})}\vert \wideparen{V}_R- \wideparen{U}\vert_{H^1(\widehat{\Omega})}^2.
     \label{eq:Uparen:best_approximation}
  \end{align}
  Now, we define for any $\wideparen{V} \in BL_{0, \sharp}(\widehat{\Omega})$
  the interpolant $\Pi_R \wideparen{V} \in BL_{\pm R, \sharp}(\widehat{\Omega})$ where
  \begin{align}
     \left(\Pi_R \wideparen{V}\right) (X,Y) = \wideparen{V}(X,Y) \cdot \left\{\begin{array}{ll} 1, & \vert Y \vert < \frac{R}{2},\\
    2\frac{R-Y}{R}, & \frac{R}{2}<\vert Y \vert < R.\end{array}\right.,
   \label{eq:Pi_R:def}
  \end{align}
  which extended by $0$ into $\widehat{\Omega} \setminus \widehat{\Omega}_R$ is in $BL_{0, \sharp}(\widehat{\Omega})$.

  To estimate the interpolation error in the $H^1(\widehat{\Omega})$-seminorm we compute the $L^2(\widehat{\Omega})$-norms of the derivatives of the difference $\Pi_R \wideparen{U} - \wideparen{U}$. As $\Pi_R \wideparen{U} = \wideparen{U}$ in $\widehat{\Omega}_{\nicefrac{R}{2}}$ the errors consists of contributions in $\widehat{\Omega}_R \setminus \widehat{\Omega}_{\nicefrac{R}{2}}$ and in $\widehat{\Omega} \setminus \widehat{\Omega}_R$.

  As the solution $\wideparen{U} \in BL_{0,\sharp}(\widehat{\Omega})$ of~\eqref{eq:var:Uparen_alpha} satisfies
  $-\Delta \wideparen{U} = 0$ for $|Y| > R_1$
  using separation of variables we can assure that
  it admits the series representation
  \begin{align}
      \widehat{U}(X,Y) &= \sum_{k=1}^\infty U_{k,\pm} \exp(2\pi k \imag X) \exp(-2\pi k (\pm Y-R_0))
      \text{ for } \pm Y > \tfrac{R}{2} \geq R_1.
      \label{eq:wideparenU:series}
   \end{align}
   Now, with the trace theorem and Lemma~\ref{lem:var:Uparen_alpha} it holds with constants $c,C$ that
   \begin{align}
      2\pi \sum_{k=1}^\infty k |U_{k,\pm}|^2 \leq c\|\wideparen{U}\|_{H^{\nicefrac12}(\Gamma_{\pm R_0})}^2
      \leq C |u_\infty|^2\ .
      \label{eq:Uparen:H12norm}
   \end{align}
   Using $\Pi_R\wideparen{U} = 0$ in $\widehat{\Omega} \setminus \widehat{\Omega}_R$ we find
  \begin{align*}
     \|\partial_X(\Pi_R\wideparen{U} - \wideparen{U})\|^2_{L^2(\widehat{\Omega} \setminus \widehat{\Omega}_R)} 
     &= 4 \pi^2 \sum_{\pm}\sum_{k=1}^\infty k^2 |U_{k,\pm}|^2 \int_R^\infty \exp(-4\pi k(Y-R_0))\,\text{d}Y \\
     &\leq \pi \sum_{\pm}\sum_{k=1}^\infty k |U_{k,\pm}|^2 \exp(-4\pi k (R-R_0))\\
     &\leq  \exp(-4\pi (R-R_0)) \pi\sum_{\pm}\sum_{k=1}^\infty k |U_{k,\pm}|^2 \\
     &\leq C \exp(-4\pi (R-R_0)) |u_\infty|^2\ ,
 \end{align*}
  and in analogy we obtain
  \begin{align*}
     \|\partial_Y(\Pi_R\wideparen{U} - \wideparen{U})\|^2_{L^2(\widehat{\Omega} \setminus \widehat{\Omega}_R)}
     &\leq C \exp(-4\pi (R-R_0)) |u_\infty|^2\ .
  \end{align*}
  To analyse the error in $\widehat{\Omega}_R \setminus \widehat{\Omega}_{\nicefrac{R}{2}}$ we first see that
  \begin{align*}
     \Pi_R\wideparen{U}(X,Y) - \wideparen{U}(X,Y) = \left( 2\frac{R - Y}{R} - 1 \right)\wideparen{U}(X,Y)
     = \left( 1 - \frac{2Y}{R}\right) \wideparen{U}(X,Y),
  \end{align*}
  and using~\eqref{eq:wideparenU:series} we find that
  \begin{align*}
     \|\partial_X(\Pi_R\wideparen{U} - \wideparen{U})\|^2_{L^2(\widehat{\Omega}_R \setminus \widehat{\Omega}_{\nicefrac{R}{2}})}
     &= 4\pi^2 \sum_{\pm}\sum_{k=1}^\infty k^2 \vert U_{k,\pm}\vert^2 E_k(R)
  \end{align*}
  where
  \begin{align*}
     E_k(R) &= \int_{\nicefrac{R}{2}}^R \left( 1 - \frac{2Y}{R}\right)^2\exp(- 4\pi k(Y - R_0))\text{d}Y \\
            &\leq \int_{\nicefrac{R}{2}}^R \exp(- 4\pi k(Y - R_0))\text{d}Y
    \leq \frac{1}{4\pi k}\exp\left(- 2\pi k\left(R - 2R_0\right)\right)\ .
  \end{align*}
  Now, inserting the inequality~\eqref{eq:Uparen:H12norm} we can assert that
  \begin{align*}
     \|\partial_X(\Pi_R\wideparen{U} - \wideparen{U})\|^2_{L^2(\widehat{\Omega}_R \setminus \widehat{\Omega}_{\nicefrac{R}{2}})}
     \leq C \exp(-2\pi (R-2R_0)) |u_\infty|^2\ .
  \end{align*}
  For the $Y$-derivative we obtain
  \begin{align*}
     \|\partial_Y(\Pi_R\wideparen{U} - \wideparen{U})\|^2_{L^2(\widehat{\Omega}_R \setminus \widehat{\Omega}_{\nicefrac{R}{2}})} &\leq
     \frac{4}{R^2} \sum_{\pm}\sum_{k=1}^\infty\vert U_{k,\pm}\vert^2\int_{\nicefrac{R}{2}}^R  \exp(- 4\pi k(Y - R_0))\text{d}Y\nonumber\\
    &\hspace{1em}+ 4\pi^2 \sum_{k=1}^\infty k^2 \vert U_{k,\pm}\vert^2 E_k(R)\\
    &\leq \sum_{\pm}\sum_{k=1}^\infty \vert U_{k,\pm} \vert^2 \left( \frac{1}{\pi k R^2} + \pi k \right)\exp(-2 \pi k(R - 2R_0))\\
     &\leq C \exp(-2\pi (R-2R_0)) |u_\infty|^2\ ,
  \end{align*}
  where we used again~\eqref{eq:Uparen:H12norm}.

  Finally, adding all the error contributions and using~\eqref{eq:Uparen:best_approximation} with $\wideparen{V}_R = \Pi_R\wideparen{U}$ we can assert the statement of the lemma.
\end{proof}



\section{Variational formulation for near and far field}\label{sec:var_coupled}

\subsection{The variational formulation on unbounded periodicity cells}

For the variational problem of the coupled problem we consider as unknown for the near field
\begin{equation}
    \wideparen{U}_{\mathrm{int}}(\vx,X,Y) := U_{\mathrm{int}}(\vx,X,Y)
    - m_\infty(\vx) - u_{\infty}(\vx)J(Y) - \alpha(\vx) Y, \quad \vx \in \Gamma,
\end{equation}
that is seeked in the variational space
\begin{align*}
L^2(\Gamma, BL_{0,\sharp}(\widehat{\Omega})) &:= \Big\{ \wideparen{V}_{\mathrm{int}}(\vx,\cdot) \in BL_{0,\sharp}(\widehat{\Omega}(\vx)) \text{ for almost all } \vx \in \Gamma, \\
&\hspace{2em}\|\wideparen{V}_{\mathrm{int}}(\cdot,X,Y)\|_{BL(\widehat{\Omega}(\cdot))} \in L^2(\Gamma) \Big\}.
\end{align*}
This space is equipped with the norm defined by
\begin{align*}
   \|\wideparen{V}_{\mathrm{int}}\|_{L^2(\Gamma, BL(\widehat{\Omega}))}^2
   := \int_\Gamma \|\wideparen{V}_{\mathrm{int}}(\vx,X,Y)\|_{BL(\widehat{\Omega}(\vx))}^2 \mathrm{d}\sigma\ .
\end{align*}
Functions in this space have
periodicity conditions in $X$, and the domain of definition for fixed $\vx$ is the local periodicity cell $\widehat{\Omega}(\vx)$.

As the far field $u_{\mathrm{ext}}$ can only be uniquely defined up to an additive constant we seek it in an Hilbert space of vanishing mean
in the support of $f$
\begin{align*}
   H^1_*(\Omega \setminus \Gamma):=\{v_{\mathrm{ext}} \in H^1(\Omega \setminus \Gamma): \int_{\Omega_f} v_{\mathrm{ext}}\, \mathrm{d}\vx = 0\}\ ,
\end{align*}
which still allows for jumps on $\Gamma$. We equip the space with the norm defined by
\begin{align*}
    \Vert v_{\mathrm{ext}} \Vert^2_{H^1_*(\Omega\setminus\Gamma)} := \vert v_{\mathrm{ext}} \vert^2_{H^1(\Omega\setminus\Gamma)} + \Vert [v_{\mathrm{ext}}] \Vert^2_{L^2(\Gamma)}\ ,
\end{align*}
and due to term $\Vert [v_{\mathrm{ext}}] \Vert^2_{L^2(\Gamma)}$ the norm is only zero for $v_{\mathrm{ext}} = 0$.

To obtain a variational formulation we consider first~\eqref{eq:var:Uparen_alpha} where $\wideparen{U}$, $\wideparen{V}$ are replaced by $\wideparen{U}_{\mathrm{int}}$ and $\wideparen{V}_{\mathrm{int}}$, the second equation~\eqref{eq:var:Uparen_alpha:2} is multiplied with $\beta \in L^2(\Gamma)$, $\alpha$ is considered in $L^2(\Gamma)$ and both equations are multiplied with $1 / \varepsilon$ and integrated over $\Gamma$. Then, multiplying~\eqref{eq:coupled_system:1} by $v_{\mathrm{ext}} \in H^1_*(\Omega \setminus \Gamma)$, integrating over $\Omega \setminus \Gamma$ and
using $[\partial_n u_\mathrm{ext}] = 0$ by~\eqref{eq:coupled_system:3} and $\{ \partial_n u_\mathrm{ext}\} = \alpha / \varepsilon$ by~\eqref{eq:coupled_system:9} and
multiplying~\eqref{eq:coupled_system:7} by $v_\infty \in L^2(\Gamma)$ and integrating over~$\Gamma$ leads to the coupled variational formulation:
Seek $(u_{\mathrm{ext}}, \alpha, \wideparen{U}_{\mathrm{int}}, u_{\infty}) \in H^1_*(\Omega \setminus \Gamma) \times L^2(\Gamma) \times  L^2(\Gamma, BL_{0,\sharp}(\widehat{\Omega})) \times L^2(\Gamma)$ such that
\begin{align}
    \int_{\Omega \setminus \Gamma} \nabla u_{\mathrm{ext}} \cdot \nabla v_{\mathrm{ext}} \mathrm{d}\vx + \frac{1}{\varepsilon}\int_{\Gamma} \alpha[v_{\mathrm{ext}}] \mathrm{d}\sigma 
    &= \int_{\Omega \setminus \Gamma} f v_{\mathrm{ext}}
    \mathrm{d}\vx  \nonumber\\
    \frac{1}{\varepsilon}\int_{\Gamma} -[u_{\mathrm{ext}}]v_\infty + u_{\infty} v_\infty \mathrm{d}\sigma &= 0 \nonumber
    \\[-1em]
    \label{eq:coupled_system:var}
    \\
    \frac{1}{\varepsilon}\int_{\Gamma}\left(\alpha  \int_{\partial\widehat{\Omega}(\vx)}  \wideparen{V}_\mathrm{int} \widehat{n}_2\mathrm{d}_{XY}\sigma + \int_{\widehat{\Omega}(\vx)}  \nabla \wideparen{U}_{\mathrm{int}} \cdot\nabla \wideparen{V}_\mathrm{int} - u_{\infty} J'' \wideparen{V}_\mathrm{int}\mathrm{d}(X,Y) \right) \mathrm{d}\sigma &= 0 \nonumber 
    \\
   \frac{1}{\varepsilon} \int_{\Gamma} \left(\alpha \beta \vert \widehat{\Omega}_w(\vx) \vert -\beta \int_{\partial \widehat{\Omega}(\vx)} \wideparen{U}_{\mathrm{int}} \widehat{n}_2 \mathrm{d}_{XY}\sigma
   - u_{\infty} \beta \right)\mathrm{d}\sigma  &= 0 \nonumber
\end{align}
for all $(v_{\mathrm{ext}}, \beta, \wideparen{V}_{\mathrm{int}}, v_{\infty}) \in H^1_*(\Omega \setminus \Gamma) \times L^2(\Gamma) \times  L^2(\Gamma, BL_{0,\sharp}(\widehat{\Omega})) \times L^2(\Gamma)$.

To discuss the well-posedness we introduce the product space $\mathcal{W} = H^1_*(\Omega \setminus \Gamma) \times L^2(\Gamma)\times L^2(\Gamma, BL_{0,\sharp}(\widehat{\Omega})) \times L^2(\Gamma)$
with $\varepsilon$-dependent norm defined by
\begin{align}
   \big\| (v_{\mathrm{ext}},\beta, \wideparen{V}_\mathrm{int},v_{\infty}) \big\|_{\mathcal{W},\varepsilon}^2
   := \big\| v_{\mathrm{ext}} \big\|_{H^1(\Omega \setminus \Gamma)}^2
   + \frac{1}{\varepsilon}\left(
    \big\| \beta \big\|_{L^2(\Gamma)}^2
   + \|\wideparen{U}_{\mathrm{int}}\|_{L^2(\Gamma, BL(\widehat{\Omega}))}^2
   + \big\| v_\infty \big\|_{L^2(\Gamma)}^2\right)
   \label{eq:norm:Wepsilon}
\end{align}
and a related seminorm
\begin{align*}
   \big| (v_{\mathrm{ext}},\beta, \wideparen{V}_\mathrm{int},v_{\infty}) \big\|_{\mathcal{W},\varepsilon}^2
   := \big| v_{\mathrm{ext}} \big|_{H^1(\Omega \setminus \Gamma)}^2
   + \frac{1}{\varepsilon}\left(
    \big\| \beta \big\|_{L^2(\Gamma)}^2
   + \|\wideparen{U}_{\mathrm{int}}\|_{L^2(\Gamma, BL(\widehat{\Omega}))}^2
    + \big\| v_\infty \big\|_{L^2(\Gamma)}^2\right)\ .
\end{align*}
On this space we define the bilinear form
\begin{align}
    \mathsf{b}((&u_{\mathrm{ext}}, \alpha, \wideparen{U}_\mathrm{int},u_{\infty}),(v_{\mathrm{ext}},\beta, \wideparen{V}_\mathrm{int},v_{\infty})) \nonumber \\
    &:= \int_{\Omega \setminus \Gamma} \nabla u_{\mathrm{ext}} \cdot \nabla v_{\mathrm{ext}} \mathrm{d}\vx
    + \frac{1}{\varepsilon}\int_{\Gamma} \alpha[v_{\mathrm{ext}}] - [u_{\mathrm{ext}}]v_\infty + u_{\infty} v_\infty \mathrm{d}\sigma
    \label{eq:coupled_system:bf}
    \\
    &\hspace{2em} + \frac{1}{\varepsilon}\int_{\Gamma}\left(\alpha  \int_{\partial\widehat{\Omega}(\vx)} \wideparen{V}_\mathrm{int} \widehat{n}_2\mathrm{d}_{XY}\sigma + \int_{\widehat{\Omega}(\vx)}  \nabla \wideparen{U}_{\mathrm{int}} \cdot\nabla \wideparen{V}_\mathrm{int} - u_{\infty} J'' \wideparen{V}_\mathrm{int}\mathrm{d}(X,Y) \right) \mathrm{d}\sigma \nonumber \\
    &\hspace{2em} +\frac{1}{\varepsilon} \int_{\Gamma} \left(\alpha \beta \vert \widehat{\Omega}_w(\vx) \vert -\beta \int_{\partial \widehat{\Omega}(\vx)} \wideparen{U}_{\mathrm{int}} \widehat{n}_2 \mathrm{d}_{XY}\sigma
   - u_{\infty} \beta \right)\mathrm{d}\sigma\ ,
   \nonumber
\end{align}
for which we state inf-sup-conditons where the first gives only a lower bound in terms of the seminorm.
\begin{lemma}[inf-sup conditions]
Let $\|J'\|_{L^\infty(\mathbb{R})} \leq \tfrac{1}{2}$.
Then, there exists a constant $\gamma > 0$ independent of $\varepsilon$
such that for all $(u_{\mathrm{ext}}, \alpha, \wideparen{U}_\mathrm{int},u_{\infty}) \in \mathcal{W}$ it holds
\begin{subequations}
  \begin{align}
     \sup_{(v_{\mathrm{ext}},\beta, \wideparen{V}_\mathrm{int},v_{\infty}) \in \mathcal{W} \setminus \{ \boldsymbol{0} \} }
     \hspace{-2.5em}
     \frac{\left| \mathsf{b}((u_{\mathrm{ext}}, \alpha, \wideparen{U}_\mathrm{int},u_{\infty}),(v_{\mathrm{ext}},\beta, \wideparen{V}_\mathrm{int},v_{\infty})) \right|}
     {\left\| (v_{\mathrm{ext}},\beta, \wideparen{V}_\mathrm{int},v_{\infty}) \right\|_{\mathcal{W},\varepsilon}}
     &\geq \gamma \left|(u_{\mathrm{ext}}, \alpha, \wideparen{U}_\mathrm{int},u_{\infty})\right|_{\mathcal{W},\varepsilon} \ ,
     \label{eq:coupled:inf_sup:1}\\
     \intertext{and for all $(v_{\mathrm{ext}},\beta, \wideparen{V}_\mathrm{int},v_{\infty}) \in \mathcal{W} \setminus \{(0,0,0,0) \}$ it holds}
     \sup_{(u_{\mathrm{ext}},\alpha, \wideparen{U}_\mathrm{int},u_{\infty}) \in \mathcal{W} \setminus \{ \boldsymbol{0} \} }
     \hspace{-2.5em}
     \left| \mathsf{b}((u_{\mathrm{ext}}, \alpha, \wideparen{U}_\mathrm{int},u_{\infty}),(v_{\mathrm{ext}},\beta, \wideparen{V}_\mathrm{int},v_{\infty})) \right|
     &> 0 \ .
  \end{align}
\end{subequations}
\label{lem:coupled:inf_sup}
\end{lemma}
\begin{proof}
   First, integrating by parts we find that
   \begin{align*}
      \int_{\partial\widehat{\Omega}(\vx)}
      \wideparen{U}_{\mathrm{int}} \widehat{n}_2 \mathrm{d}\sigma_{XY}
      &= \int_{\widehat{\Omega}(\vx)} Y'
      \partial_Y \wideparen{U}_{\mathrm{int}} \mathrm{d}(X,Y)
      = \int_{\widehat{\Omega}(\vx)}\partial_Y \wideparen{U}_{\mathrm{int}} \mathrm{d}(X,Y),
      \\
      -\int_{\widehat{\Omega}(\vx)} J''\wideparen{U}_{\mathrm{int}} \mathrm{d}(X,Y)
      &= \int_{\widehat{\Omega}(\vx)} J' \partial_Y \wideparen{U}_{\mathrm{int}} \mathrm{d}(X,Y),
   \end{align*}
   since the boundary term $\int_{\partial\widehat{\Omega}(\vx)} J'\wideparen{U}_{\mathrm{int}} \widehat{n}_2 \mathrm{d}\sigma_{XY}$ vanishes as $J'$ is zero on the wall boundary $\partial\widehat{\Omega}_w$.
   Inserting the test function
   \begin{align*}
      (v_{\mathrm{ext}},\beta, \wideparen{V}_\mathrm{int},v_{\infty})
      &= \left(u_{\mathrm{ext}},
       \alpha - \tfrac{\sqrt{2}}{2} u_\infty, \wideparen{U}_\mathrm{int},\alpha
      \right)\ .
   \end{align*}
   into the bilinear form, the mixed terms with $\alpha$ and $u_\infty$ cancel out.
   Now, defining $m_w := \inf_{\vx \in \Gamma}\vert \widehat{\Omega}_w(\vx)\vert$, using the above formulas and
   $m_w \leq \vert \widehat{\Omega}_w(\vx)|$
   we obtain
   \begin{align*}
         \mathsf{b}\bigg(&(u_{\mathrm{ext}},\alpha, \wideparen{U}_{\mathrm{int}},u_{\infty}),
         \left(u_{\mathrm{ext}},
      \alpha - \tfrac{\sqrt{2}}{2} u_\infty, \wideparen{U}_\mathrm{int},\alpha
      \right) \bigg)\\
         &\geq\vert u_{\mathrm{ext}}\vert^2_{H^1(\Omega\setminus\Gamma)}+ \frac{1}{\varepsilon}\int_{\Gamma}\vert \wideparen{U}_{\mathrm{int}}\vert^2_{H^1(\widehat{\Omega}(\vx))} \mathrm{d}\sigma + \frac{m_w}{\varepsilon} \| \alpha\|_{L^2(\Gamma)}^2\\
     &\hspace{1em}+\frac{\sqrt{2}}{2\,\varepsilon} \|u_{\infty}\|^2_{L^2(\Gamma)}
     +\frac{1}{\varepsilon}\int_{\Gamma}u_{\infty}\int_{\widehat{\Omega}(\vx)} \left(J' + \tfrac{\sqrt{2}}{2}\right)\partial_Y\wideparen{U}_{\mathrm{int}} \mathrm{d}(X,Y)\mathrm{d}\sigma.
\end{align*}
Using Young's inequality we find that
\begin{multline*}
   \int_{\Gamma}u_{\infty}\int_{\widehat{\Omega}(\vx)} \left(J' + \tfrac{\sqrt{2}}{2}\right)\partial_Y\wideparen{U}_{\mathrm{int}} \mathrm{d}(X,Y)\mathrm{d}\sigma 
   \geq
   -\frac{3}{4}\int_{\Gamma} |\wideparen{U}_{\mathrm{int}} |_{H^1(\widehat{\Omega}(\vx))}^2 \mathrm{d}\sigma
   - \frac{1}{3} \left( \|J'\|_{L^\infty(\mathbb{R})} + \tfrac{\sqrt{2}}{2}
   \right)^2
   \| u_\infty\|^2_{L^2(\Gamma)}\,.
\end{multline*}
With the assumption on $\|J'\|_{L^\infty(\mathbb{R})}$ we can assert that
\begin{align*}
  1 - \frac{1}{3}\left(\|J'\|_{L^\infty(\mathbb{R})}
   + \tfrac{\sqrt{2}}{2}\right)^2 \geq \frac{1}{3}
\end{align*}
and therefore
\begin{align*}
   \mathsf{b}\bigg(&(u_{\mathrm{ext}},\alpha, \wideparen{U}_{\mathrm{int}},u_{\infty}),
         \left(u_{\mathrm{ext}},
       \alpha - \tfrac{\sqrt{2}}{2} u_\infty, \wideparen{U}_\mathrm{int}, \alpha
      \right) \bigg)\\
         &\geq\vert u_{\mathrm{ext}}\vert^2_{H^1(\Omega\setminus\Gamma)}+ \frac{1}{\varepsilon}
         \left(\frac{1}{4}\int_{\Gamma}\vert \wideparen{U}_{\mathrm{int}}\vert^2_{H^1(\widehat{\Omega}(\vx))} \mathrm{d}\sigma + m_w \| \alpha\|_{L^2(\Gamma)}^2
     + \frac{1}{3} \| u_{\infty}\|_{L^2(\Gamma)}^2\right) \\
     &\geq \gamma \sqrt{\vert u_{\mathrm{ext}}\vert^2_{H^1(\Omega\setminus\Gamma)}+ \frac{1}{\varepsilon}
         \left(\int_{\Gamma}\vert \wideparen{U}_{\mathrm{int}}\vert^2_{H^1(\widehat{\Omega}(\vx))} \mathrm{d}\sigma + \| \alpha\|_{L^2(\Gamma)}^2
     + \| u_{\infty}\|_{L^2(\Gamma)}^2\right)}\\
    &\hspace{1em}
    \cdot \sqrt{\vert u_{\mathrm{ext}}\vert^2_{H^1(\Omega\setminus\Gamma)}+ \frac{1}{\varepsilon}
         \left(\int_{\Gamma}\vert \wideparen{U}_{\mathrm{int}}\vert^2_{H^1(\widehat{\Omega}(\vx))} \mathrm{d}\sigma + \| \alpha - \tfrac{\sqrt{2}}{2}u_\infty\|_{L^2(\Gamma)}^2
     + \| \alpha \|_{L^2(\Gamma)}^2\right)}
\end{align*}
for some well-chosen $\gamma > 0$ only depending on $m_w$ since
\begin{align*}
\| \alpha - \tfrac{\sqrt{2}}{2}u_\infty\|^2_{L^2(\Gamma)} + \|\alpha\|^2_{L^2(\Gamma)} \leq 2 \left( \| \alpha \|^2_{L^2(\Gamma)} + \| u_\infty\|^2_{L^2(\Gamma)} \right).
\end{align*}
As the $H^1(\widehat{\Omega}(\vx))$-seminorm and the $BL(\widehat{\Omega}(\vx))$-norm are equivalent on $BL_{0,\sharp}(\widehat{\Omega}(\vx))$ by Lemma~\ref{lem:BL}
the inequality~\eqref{eq:coupled:inf_sup:1} follows.

Now, we aim to show the second inf-sup condition. For this we fix the test functions and choose appropriate trial functions. More precisely, we see that for
$M_w := \max(4, \frac{1}{3} \sup_{\vx \in \Gamma}\vert \widehat{\Omega}_w(\vx)\vert)$ it holds
\begin{align*}
         \mathsf{b}\bigg(&
         \left(v_{\mathrm{ext}}, v_\infty, 2 M_w \wideparen{V}_\mathrm{int},
      M_w (v_\infty -  2\beta)\right),
      \left(v_{\mathrm{ext}},\beta, \wideparen{V}_{\mathrm{int}},v_{\infty}
         \right)\bigg)\\
         &= \vert v_{\mathrm{ext}}\vert^2_{H^1(\Omega\setminus\Gamma)}
         + \frac{2M_w}{\varepsilon}\int_{\Gamma}\vert \wideparen{V}_{\mathrm{int}}\vert^2_{H^1(\widehat{\Omega}(\vx))} \mathrm{d}\sigma
         + \frac{M_w}{\varepsilon} \| v_\infty \|_{L^2(\Gamma)}^2
         + \frac{2M_w}{\varepsilon} \| \beta \|^2_{L^2(\Gamma)}\\
     &\hspace{1em}
     + \frac{1}{\varepsilon} \int_\Gamma \int_{\widehat{\Omega}(\vx)}  \left(  (1 + 2M_w J') v_\infty -  2 M_w  (1 + J') \beta \right)\partial_Y \wideparen{V}_{\mathrm{int}} \text{d}(X,Y) \,\text{d}\sigma \\
     &\hspace{1em}
     + \frac{1}{\varepsilon}\int_\Gamma  (|\widehat{\Omega}_w| -3 M_w)\beta v_\infty \text{d}\sigma\ .
     \intertext{Using Young's inequality, the definition of $M_w$ and the assumption on $J'$ we can estimate the mixed terms and using again
     Poincar\'e estimate of Lemma~\ref{lem:BL} we obtain}
          \mathsf{b}\bigg(&
         \left(v_{\mathrm{ext}}, v_\infty, 2 M_w \wideparen{V}_\mathrm{int},
      M_w (v_\infty -  2\beta)\right),
      \left(v_{\mathrm{ext}},\beta, \wideparen{V}_{\mathrm{int}},v_{\infty}
         \right)\bigg)\\
     &\geq \vert v_{\mathrm{ext}}\vert^2_{H^1(\Omega\setminus\Gamma)} \\
     &\hspace{1em}
     + \frac{1}{\varepsilon}
     \left( 2M_w - \tfrac{1}{2} (1 + M_w \|J'\|_{L^\infty(\mathbb{R})})
     - M_w( 1 + \|J'\|_{L^\infty(\mathbb{R})}) \right)
     \int_{\Gamma}\vert \wideparen{V}_{\mathrm{int}}\vert^2_{H^1(\widehat{\Omega}(\vx))} \mathrm{d}\sigma \\
       &\hspace{1em}
        + \frac{1}{\varepsilon} \left( M_w - \tfrac{1}{2} (1 + M_w \|J'\|_{L^\infty(\mathbb{R})}) - \tfrac12 (|\widehat{\Omega}_w| -3 M_w)
        \right) \| v_\infty \|_{L^2(\Gamma)}^2 \\
        &\hspace{1em}
        + \frac{1}{\varepsilon} \left( 2M_w - M_w (1 + \|J'\|_{L^\infty(\mathbb{R})}) - \tfrac12 (|\widehat{\Omega}_w| -3 M_w)
        \right) \| \beta \|_{L^2(\Gamma)}^2 \\
    &\geq \vert v_{\mathrm{ext}}\vert^2_{H^1(\Omega\setminus\Gamma)}
       + \frac{1}{2\,\varepsilon}
     \int_{\Gamma}\vert \wideparen{V}_{\mathrm{int}}\vert^2_{H^1(\widehat{\Omega}(\vx))} \mathrm{d}\sigma
        + \frac{11}{\varepsilon} \| v_\infty \|_{L^2(\Gamma)}^2
        + \frac{2}{\varepsilon}  \| \beta \|_{L^2(\Gamma)}^2 \\
      &\geq \vert v_{\mathrm{ext}}\vert^2_{H^1(\Omega\setminus\Gamma)}
        + \frac{1}{2\,\varepsilon(C_p+1)}
      \| \wideparen{V}_{\mathrm{int}}\|_{L^2(\Gamma, BL(\widehat{\Omega}(\vx))}^2
         + \frac{11}{\varepsilon} \| v_\infty \|_{L^2(\Gamma)}^2
         + \frac{2}{\varepsilon}  \| \beta \|_{L^2(\Gamma)}^2\ ,
\end{align*}
which is positive if $\left(v_{\mathrm{ext}},\beta, \wideparen{V}_{\mathrm{int}},v_{\infty}
         \right) \in \mathcal{W} \setminus \{ (0,0,0,0) \}$.
\end{proof}

\begin{theorem}[Well-posedness]
  Let $\|J'\|_{L^\infty(\mathbb{R})} \leq \tfrac{1}{2}$ and let $\varepsilon_0 > 0$. Then, for all $\varepsilon \in (0, \varepsilon_0]$ the variational formulation~\eqref{eq:coupled_system:var} admits a unique solution and
  there exists a constant $C > 0$ independent of $\varepsilon$ such that
  \begin{equation}
   \vert u_{\mathrm{ext}} \vert_{H^1(\Omega\setminus\Gamma)}
   + \frac{1}{\sqrt{\varepsilon}} \left(\Vert \alpha\Vert_{L^2(\Gamma)}
   + \|\wideparen{U}_{\mathrm{int}}\|_{L^2(\Gamma, BL(\widehat{\Omega}))}
   +
   \Vert u_{\infty}\Vert_{L^2(\Gamma)} \right) \leq C \Vert f \Vert_{L^2(\Omega\setminus\Gamma)}\ .
   \label{eq:coupled_system:stability}
\end{equation}
  \label{lem:coupled_system:stability}
\end{theorem}
\begin{proof}
   The proof is divided into two steps. First, we show that a solution of~\eqref{eq:coupled_system:var} is bounded by the right hand side $f$ and therefore it is unique. Second, we use the Fredholm theory to conclude that a solution exists for any $f \in L^2(\Omega \setminus \Gamma)$.

   Using the definition of the bilinear form $\mathsf{b}$ we see that the solution $(u_{\mathrm{ext}},\alpha, \wideparen{U}_\mathrm{int},u_{\infty}) \in \mathcal{W}$ of~\eqref{eq:coupled_system:var} satisfies for all $(v_{\mathrm{ext}},\beta, \wideparen{V}_\mathrm{int},v_{\infty}) \in \mathcal{W}$
   \begin{align*}
      \mathsf{b}\left((u_{\mathrm{ext}},\alpha, \wideparen{U}_\mathrm{int},u_{\infty}), (v_{\mathrm{ext}},\beta, \wideparen{V}_\mathrm{int},v_{\infty})\right) = \int_{\Omega \setminus \Gamma} f v_{\mathrm{ext}} \,\text{d}\vx\ .
   \end{align*}
   Moreover, in view of the second equation of~\eqref{eq:coupled_system:var} we can assert that
   \begin{align}
      u_{\infty} = [u_{\mathrm{ext}}] \in H^{\nicefrac{1}{2}}(\Gamma)\ .
      \label{eq:uinf_equal_jump_uext}
   \end{align}
  Using~\eqref{eq:coupled:inf_sup:1}, the Cauchy-Schwarz inequality and~\eqref{eq:uinf_equal_jump_uext} we find for any $\varepsilon_0 > 0$ and $\varepsilon \leq \varepsilon_0$ that
  \begin{align*}
    \frac{1}{\gamma} \|f &\|_{L^2(\Omega \setminus \Gamma)} \big\| u_{\mathrm{ext}} \big\|_{L^2(\Omega \setminus \Gamma)} \\
    &\geq
     \big| u_{\mathrm{ext}} \big|_{H^1(\Omega \setminus \Gamma)}^2
     + \frac{1}{\varepsilon}\left(
    \big\| \alpha \big\|_{L^2(\Gamma)}^2
   + \|\wideparen{U}_{\mathrm{int}}\|_{L^2(\Gamma, BL(\widehat{\Omega}))}^2
    + \big\| u_\infty \big\|_{L^2(\Gamma)}^2 \right) \\
    &\geq
    \big| u_{\mathrm{ext}} \big|_{H^1(\Omega \setminus \Gamma)}^2
    + \tfrac{1}{2\varepsilon_0} \| [u_{\mathrm{ext}}] \|_{L^2(\Gamma)}^2
   + \tfrac{1}{\varepsilon}\big\| \alpha \big\|_{L^2(\Gamma)}^2
   + \tfrac{1}{\varepsilon}\|\wideparen{U}_{\mathrm{int}}\|_{L^2(\Gamma, BL(\widehat{\Omega}))}^2
    + \tfrac{1}{2\varepsilon} \big\| u_\infty \big\|_{L^2(\Gamma)}^2 \\
    &\geq
    \min(1, \tfrac{1}{2\varepsilon_0}) \Vert u_{\mathrm{ext}} \Vert^2_{H^1_*(\Omega\setminus\Gamma)}
    + \tfrac{1}{2\varepsilon}\left(\big\| \alpha \big\|_{L^2(\Gamma)}^2
   + \|\wideparen{U}_{\mathrm{int}}\|_{L^2(\Gamma, BL(\widehat{\Omega}))}^2
    + \big\| u_\infty \big\|_{L^2(\Gamma)}^2\right)\ .
  \end{align*}
  Using the Poincar\'{e} inequality for functions in $H^1_*(\Omega\setminus\Gamma)$
  \begin{align*}
     \big\| v_{\mathrm{ext}} \big\|_{L^2(\Omega \setminus \Gamma)}
     \leq C_P \big\| v_{\mathrm{ext}} \big\|_{H^1_*(\Omega\setminus\Gamma)} \quad \text{ for all } v_{\mathrm{ext}} \in H^1_*(\Omega\setminus\Gamma)
  \end{align*}
  and using Young's inequality we obtain
  \begin{multline*}
     \min(1, \tfrac{1}{2\varepsilon_0})\Vert u_{\mathrm{ext}} \Vert^2_{H^1_*(\Omega\setminus\Gamma)}
    + \tfrac{1}{2\varepsilon}\left(\big\| \alpha \big\|_{L^2(\Gamma)}^2
   + \|\wideparen{U}_{\mathrm{int}}\|_{L^2(\Gamma, BL(\widehat{\Omega}))}^2
    +  \big\| u_\infty \big\|_{L^2(\Gamma)}^2 \right)
    \leq \frac{C_P^2}{2\gamma^2} \|f \|_{L^2(\Omega \setminus \Gamma)}^2\ ,
  \end{multline*}
  and~\eqref{eq:coupled_system:stability} follows.

  Now, we define the sesquilinearform
  \begin{multline*}
     \mathsf{b}_0\left((u_{\mathrm{ext}},\alpha, \wideparen{U}_\mathrm{int},u_{\infty}), (v_{\mathrm{ext}},\beta, \wideparen{V}_\mathrm{int},v_{\infty})\right) 
     :=
     \mathsf{b}\left((u_{\mathrm{ext}},\alpha, \wideparen{U}_\mathrm{int},u_{\infty}), (v_{\mathrm{ext}},\beta, \wideparen{V}_\mathrm{int},v_{\infty})\right)
     + \int_\Gamma [u_{\mathrm{ext}}] [v_{\mathrm{ext}}] \,\text{d}\sigma
  \end{multline*}
  for which corresponding inf-sup-conditions with
  the norm $\|\cdot \|_{\mathcal{W},\varepsilon}$ as defined in~\eqref{eq:norm:Wepsilon} holds.
  Hence, the associated operator $\mathsf{B}_0: \mathcal{W} \to \mathcal{W}$ is an isomorphism.
  Moreover, the operator $\mathsf{K}: \mathcal{W} \to \mathcal{W}$ defined by
  \begin{align*}
      \left(\mathsf{K}(u_{\mathrm{ext}},\alpha, \wideparen{U}_\mathrm{int},u_{\infty}), (v_{\mathrm{ext}},\beta, \wideparen{V}_\mathrm{int},v_{\infty})\right)_{\mathcal{W},\varepsilon}
      = -\int_\Gamma [u_{\mathrm{ext}}] [v_{\mathrm{ext}}] \,\text{d}\sigma
      \quad \forall (v_{\mathrm{ext}},\beta, \wideparen{V}_\mathrm{int},v_{\infty}) \in \mathcal{W}
  \end{align*}
with the corresponding inner product $(\cdot,\cdot)_{\mathcal{W},\varepsilon}$
is compact as the trace space $H^{\nicefrac{1}{2}}(\Gamma)$ of $H^1_\star(\Omega \setminus \Gamma)$ is compactly embedded in $L^2(\Gamma)$ due to theorem of Rellich-Kondrachov. Hence, the operator $\mathsf{B} = \mathsf{B}_0 + \mathsf{K}$ corresponding to the sesquilinear form $\mathsf{b}$ of the variational formulation~\eqref{eq:coupled_system:var} is a Fredholm operator of index~$0$. Hence, by the Fredholm alternative we can conclude from the uniqueness of a solution of~\eqref{eq:coupled_system:var}, which we have shown above, its existence.

This completes the proof.
\end{proof}

\begin{remark}
   The condition $\|J'\|_{L^\infty(\mathbb{R})} \leq \frac12$ is fulfilled for example for the piecewise polynomial
   \begin{align*}
      J(Y) &= \begin{cases}
                 \frac{\operatorname{sgn}(Y)}{32}(|Y|-R_0)^3 \left( 3 (|Y| - R_0)^2 - 15 (|Y| - R_0) + 20\right), & R_0 \leq |Y| < R_0 + 2,\\[0.3em]
                 \frac{\operatorname{sgn}(Y)}{2}, &
                 |Y| \geq R_0 + 2,\\[0.3em]
                 0, & \text{otherwise},
              \end{cases}
   \end{align*}
   which is in $C^2(\mathbb{R})$ for any $R_0 > 0$, and it holds $\|J'\|_{L^\infty(\mathbb{R})} = \frac{15}{32}$.
\end{remark}

\subsection{Coupled formulation with truncated periodicity cells}

Now, as in Sec.~\ref{AbschneidefehlerInnen} the near field function shall be truncated at $Y=\pm R$ for some $R>R_1$, as this would simplify a numerical discretization of the formulation. Let $\widehat{\Omega}_R(\vx) = \widehat{\Omega}(\vx) \cap  [0,1]\times [-R,R]$ the truncated periodicity cell for each $\vx \in \Gamma$. The truncated solution will be seeked in the space
\begin{align*}
  L^2(\Gamma, BL_{0,\sharp}(\widehat{\Omega}_R)) &:= \Big\{ \wideparen{V}_{\mathrm{int}}(\vx,\cdot,\cdot) \in BL_{0,\sharp}(\widehat{\Omega}_R(\vx)) \text{ for almost all } \vx \in \Gamma, \\
&\hspace{2em}\|\wideparen{V}_{\mathrm{int}}(\cdot,X,Y)\|_{BL(\widehat{\Omega}_R(\cdot))} \in L^2(\Gamma),
\wideparen{V}_{\mathrm{int}}(\cdot,\cdot, \pm R) = 0
\Big\},
\end{align*}
which is equipped with the $\|\cdot\|_{L^2(\Gamma, BL(\widehat{\Omega}))}$-norm.

Then, the coupled variational formulation with the truncated periodicity cells is:
Seek $(u_{\mathrm{ext},R}, \alpha_R, \wideparen{U}_{\mathrm{int},R}, u_{\infty,R}) \in H^1_*(\Omega \setminus \Gamma) \times L^2(\Gamma) \times  L^2(\Gamma, BL_{0,\sharp}(\widehat{\Omega}_R)) \times L^2(\Gamma)$ such that
\begin{align}
\int_{\Omega \setminus \Gamma} \nabla u_{\textrm{ext},R} \cdot \nabla v_{\mathrm{ext},R} \mathrm{d}\vx + \frac{1}{\varepsilon}\int_{\Gamma} \alpha_R[v_{\mathrm{ext},R}] \mathrm{d}\sigma
&= \int_{\Omega \setminus \Gamma} f v_{\mathrm{ext},R} \mathrm{d}\vx  \nonumber\\
    \int_\Gamma -[u_{\textrm{ext},R}]v_{\infty,R} + u_{\infty,R} v_{\infty,R} \mathrm{d}\sigma &= 0 \nonumber\\
    \frac{1}{\varepsilon}\int_\Gamma\bigg(\alpha_R  \int_{\partial\widehat{\Omega}(x)} \wideparen{V}_{\mathrm{int},R} \widehat{n}_2  \mathrm{d}_{XY}\sigma \hspace{12em} \label{eq:coupled_system:truncated:var}\\[-0.5em]
    + \int_{\widehat{\Omega}_R(x)}\nabla \wideparen{U}_{\mathrm{int},R} \cdot \nabla \wideparen{V}_{\mathrm{int},R} - u_{\infty,R} J''\wideparen{V}_{\mathrm{int},R} \mathrm{d}(X,Y) \bigg) \mathrm{d}\sigma &= 0\nonumber\\
   \frac{1}{\varepsilon} \int_{\Gamma} \left(\alpha_R \beta_R \vert \widehat{\Omega}_w(\vx) \vert -\beta_R \int_{\partial \widehat{\Omega}(\vx)} \wideparen{U}_{\mathrm{int},R} \widehat{n}_2 \mathrm{d}_{XY}\sigma
   - u_{\infty,R} \beta_R \right)\mathrm{d}\sigma  &= 0 \nonumber
\end{align}
for all $(v_{\mathrm{ext},R}, \beta_R, \wideparen{V}_{\mathrm{int},R}, v_{\infty,R}) \in H^1_*(\Omega \setminus \Gamma) \times L^2(\Gamma) \times  L^2(\Gamma, BL_{0,\sharp}(\widehat{\Omega}_R)) \times L^2(\Gamma)$.

To discuss the well-posedness we introduce the product space $\mathcal{W}_R = H^1_*(\Omega \setminus \Gamma) \times L^2(\Gamma)\times L^2(\Gamma, BL_{0,\sharp}(\widehat{\Omega}_R)) \times L^2(\Gamma)$ equipped with the norm defined in~\eqref{eq:norm:Wepsilon}, and we define the bilinear form $\mathsf{b}_R$ as in~\eqref{eq:coupled_system:bf} where $\widehat{\Omega}(\vx)$ is replaced by $\widehat{\Omega}_R(\vx)$.

\begin{lemma}[inf-sup conditions]
Let $\|J'\|_{L^\infty(\mathbb{R})} \leq \tfrac{1}{2}$.
Then there exists a constant $\gamma > 0$ independent of $\varepsilon$ and $R$
such that for all $(u_{\mathrm{ext},R}, \alpha_R, \wideparen{U}_{\mathrm{int},R},u_{\infty,R}) \in \mathcal{W}_R$ it holds
\begin{subequations}
  \begin{align}
     \sup_{(v_{\mathrm{ext},R},\beta_R, \wideparen{V}_{\mathrm{int},R},v_{\infty,R}) \in \mathcal{W}_R \setminus \{ \boldsymbol{0} \} }
     \hspace{-2.5em}
     \frac{\left| \mathsf{b}_R((u_{\mathrm{ext},R}, \alpha_R, \wideparen{U}_{\mathrm{int},R},u_{\infty,R}),(v_{\mathrm{ext},R},\beta_R, \wideparen{V}_{\mathrm{int},R},v_{\infty,R})) \right|}
     {\left\| (v_{\mathrm{ext},R},\beta_R, \wideparen{V}_{\mathrm{int},R},v_{\infty,R}) \right\|_{\mathcal{W},\varepsilon}} \quad&
     \label{eq:coupled:truncated:inf_sup:1}
     \\
     \geq \gamma \left|(u_{\mathrm{ext},R}, \alpha_R, \wideparen{U}_{\mathrm{int},R},u_{\infty,R})\right|_{\mathcal{W},\varepsilon}& \ ,
     \nonumber \\
     \intertext{and for all $(v_{\mathrm{ext},R},\beta_R, \wideparen{V}_{\mathrm{int},R},v_{\infty},R) \in \mathcal{W}_R \setminus \{(0,0,0,0) \}$ it holds}
     \sup_{(u_{\mathrm{ext},R},\alpha_R, \wideparen{U}_{\mathrm{int},R}, u_{\infty,R}) \in \mathcal{W}_R \setminus \{ \boldsymbol{0} \} }
     \hspace{-2.5em}
     \left| \mathsf{b}_R((u_{\mathrm{ext},R}, \alpha_R, \wideparen{U}_{\mathrm{int},R},u_{\infty,R}),(v_{\mathrm{ext},R},\beta_R, \wideparen{V}_{\mathrm{int},R},v_{\infty,R})) \right|
     &> 0 \ .
  \end{align}
\end{subequations}
\label{lem:coupled:truncated:inf_sup}
\end{lemma}
\begin{proof}
   First integrating by parts we find that
   \begin{align*}
      \int_{\partial\widehat{\Omega}(\vx)}
      \wideparen{U}_{\mathrm{int},R} \widehat{n}_2 \mathrm{d}\sigma_{XY}
      &= \int_{\widehat{\Omega}_R(\vx)} Y'
      \partial_Y \wideparen{U}_{\mathrm{int},R} \mathrm{d}(X,Y)
      = \int_{\widehat{\Omega}_R(\vx)}\partial_Y \wideparen{U}_{\mathrm{int},R} \mathrm{d}(X,Y),
      \\
      -\int_{\widehat{\Omega}_R(\vx)} J''\wideparen{U}_{\mathrm{int},R} \mathrm{d}(X,Y)
      &= \int_{\widehat{\Omega}_R(\vx)} J' \partial_Y \wideparen{U}_{\mathrm{int},R} \mathrm{d}(X,Y)
   \end{align*}
   as the boundary terms on $[0,1]\times \{\pm R\}$ vanish where $\wideparen{U}_{\mathrm{int},R} = 0$ as $\wideparen{U}_{\mathrm{int},R} \in BL_{0,\sharp}
   (\widehat{\Omega}_R)$
   and the boundary term $\int_{\partial\widehat{\Omega}(\vx)} J'\wideparen{U}_{\mathrm{int},R} \widehat{n}_2 \mathrm{d}\sigma_{XY}$ vanishes as $J'$ is zero on the wall boundary~$\partial\widehat{\Omega}_w$.

   The remainder of the proof is in analogy to the one of Lemma~\ref{lem:coupled:inf_sup}.
\end{proof}

\begin{theorem}[Well-posedness]
Let $\|J'\|_{L^\infty(\mathbb{R})} \leq \frac12$ and let $\varepsilon_0 > 0$. Then, for all $\varepsilon \in (0, \varepsilon_0]$ the variational formulation~\eqref{eq:coupled_system:truncated:var}
admits a unique solution and there is a constant $C > 0$ independent of $\varepsilon$ such that
\begin{align}
  \vert u_{\mathrm{ext},R} \vert_{H^1(\Omega\setminus\Gamma)}
   + \frac{1}{\sqrt{\varepsilon}} \left(\Vert \alpha_R\Vert_{L^2(\Gamma)}
   + \|\wideparen{U}_{\mathrm{int},R}\|_{L^2(\Gamma, BL(\widehat{\Omega}))}
   +
   \Vert u_{\infty,R}\Vert_{L^2(\Gamma)} \right) \leq C \Vert f \Vert_{L^2(\Omega\setminus\Gamma)}\ .
\end{align}
\end{theorem}
\begin{proof}
  The proof is in analogy to the one of Theorem~\ref{lem:coupled_system:stability} using the inf-sup conditions of Lemma~\ref{lem:coupled:truncated:inf_sup}, where the operator $\mathsf{B}_R: \mathcal{W}_R \to \mathcal{W}_R$ associated to the bilinear form $\mathsf{b}_R$ is Fredholm of index $0$.
\end{proof}

Now, we are going to estimate the truncation error. For this we denote by $\wideparen{U}_{\mathrm{int},R}$ also the extension of $\wideparen{U}_{\mathrm{int},R}$ by $0$ onto $\widehat{\Omega}(\vx) \setminus \widehat{\Omega}_R(\vx)$ for any $\vx \in \Gamma$.

\begin{theorem}[Truncation error]\label{CeaGekoppelt} %
For the solution $(u_{\mathrm{ext}},\alpha,\wideparen{U}_{\mathrm{int}},u_{\infty})\in  H^1_*(\Omega \setminus \Gamma) \times L^2(\Gamma) \times L^2(\Gamma, BL_{0,\sharp}(\widehat{\Omega}))  \times L^2(\Gamma)$ of \eqref{eq:coupled_system:var} and the solution $(u_{\textrm{ext},R},\alpha_R, \wideparen{U}_{\mathrm{int},R},u_{\infty,R})\in  H^1_*(\Omega \setminus \Gamma) \times L^2(\Gamma) \times L^2(\Gamma, BL_{0,\sharp}(\widehat{\Omega}_R)) \times L^2(\Gamma)$ of \eqref{eq:coupled_system:truncated:var} it holds
\begin{multline}
    \left| u_{\mathrm{ext},R}-u_{\mathrm{ext}} \right|_{H^1_*(\Omega\setminus\Gamma)}
    + \frac{1}{\sqrt{\varepsilon}} \left( \Vert \alpha_R - \alpha\Vert_{L^2(\Gamma)}
    +
    \Vert \wideparen{U}_{\mathrm{int},R}- \wideparen{U}_{\mathrm{int}}\Vert_{L^2(\Gamma, BL(\widehat{\Omega}))}
    + \Vert u_{\infty,R} - u_{\infty} \Vert_{L^2(\Gamma)}\right) \\
    \leq \frac{C}{\sqrt{\varepsilon}} \exp(-\pi R)
    \label{eq:coupled_system:truncated:err}
\end{multline}
where the constant $C > 0$ is independent of $\varepsilon$ and $R$.
\end{theorem}

\begin{proof}
  Using the triangle inequality we can assert that for any  $(v_{\mathrm{ext},R},\beta_R, \wideparen{V}_{\mathrm{int},R},v_{\infty,R}) \in \mathcal{W}_R$
  \begin{multline*}
     \left| u_{\textrm{ext},R}-u_{\mathrm{ext}} \right|_{H^1_*(\Omega\setminus\Gamma)}
    + \frac{1}{\sqrt{\varepsilon}} \left( \Vert \alpha_R - \alpha\Vert_{L^2(\Gamma)}
    +
    \Vert \wideparen{U}_{\mathrm{int},R}- \wideparen{U}_{\mathrm{int}}\Vert_{L^2(\Gamma, BL(\widehat{\Omega}))}
    + \Vert u_{\infty,R} - u_{\infty} \Vert_{L^2(\Gamma)}\right) \\
    \leq  \left| v_{\textrm{ext},R}-u_{\mathrm{ext}} \right|_{H^1_*(\Omega\setminus\Gamma)} + \left| u_{\textrm{ext},R}-v_{\mathrm{ext},R} \right|_{H^1_*(\Omega\setminus\Gamma)} \hspace{12em}\\
    + \frac{1}{\sqrt{\varepsilon}} \bigg(
    \Vert \beta_R - \alpha\Vert_{L^2(\Gamma)} + \Vert \alpha_R - \beta_R\Vert_{L^2(\Gamma)}
    +
    \Vert \wideparen{V}_{\mathrm{int},R}- \wideparen{U}_{\mathrm{int}}\Vert_{L^2(\Gamma, BL(\widehat{\Omega}))} \\[-0.7em]
    +
    \Vert \wideparen{U}_{\mathrm{int},R}- \wideparen{V}_{\mathrm{int},R}\Vert_{L^2(\Gamma, BL(\widehat{\Omega}))}
    + \Vert v_{\infty,R} - u_{\infty} \Vert_{L^2(\Gamma)}
    + \Vert u_{\infty,R} - v_{\infty,R} \Vert_{L^2(\Gamma)}
    \bigg).
  \end{multline*}
  With the inf-sup conditions in Lemma~\ref{lem:coupled:inf_sup} and Galerkin orthogonality we find that
  \begin{multline*}
     \left| u_{\textrm{ext},R}-v_{\mathrm{ext},R} \right|_{H^1_*(\Omega\setminus\Gamma)}
    + \frac{1}{\sqrt{\varepsilon}} \bigg(
    \Vert \alpha_R - \beta_R\Vert_{L^2(\Gamma)}
    +
    \Vert \wideparen{U}_{\mathrm{int},R}- \wideparen{V}_{\mathrm{int},R}\Vert_{L^2(\Gamma, BL(\widehat{\Omega}))}
    + \Vert u_{\infty,R} - v_{\infty,R} \Vert_{L^2(\Gamma)}
    \bigg) \\
    \leq \frac{1}{\gamma} \sup_{(w_{\textrm{ext},R}, \delta_R, \wideparen{W}_R, w_\infty) \in \mathcal{W}_R \setminus \{ \mathbf{0} \} }
    \hspace{-3em}
    \frac{\left| \mathsf{b}((u_{\textrm{ext},R}-v_{\mathrm{ext},R}, \alpha_R - \beta_R, \wideparen{U}_{\mathrm{int},R}- \wideparen{V}_{\mathrm{int},R}, u_{\infty,R} - v_{\infty,R}), (w_{\textrm{ext},R}, \delta_R, \wideparen{W}_R, w_\infty))\right|}{\left\| (w_{\textrm{ext},R}, \delta_R, \wideparen{W}_R, w_\infty)\right\|_{\mathcal{W},\varepsilon}} \\
    \leq \frac{1}{\gamma} \sup_{(w_{\textrm{ext},R}, \delta_R, \wideparen{W}_R, w_\infty) \in \mathcal{W}_R \setminus \{ \mathbf{0} \} }
    \hspace{-3em}
    \frac{\left| \mathsf{b}((u_{\textrm{ext}}-v_{\mathrm{ext},R}, \alpha - \beta_R, \wideparen{U}_{\mathrm{int}}- \wideparen{V}_{\mathrm{int},R}, u_{\infty} - v_{\infty,R}), (w_{\textrm{ext},R}, \delta_R, \wideparen{W}_R, w_\infty))\right|}{\left\| (w_{\textrm{ext},R}, \delta_R, \wideparen{W}_R, w_\infty)\right\|_{\mathcal{W},\varepsilon}} \ .
  \end{multline*}
Applying the Cauchy-Schwarz inequality we can assert with a constant $C > 0$ that
\begin{multline*}
    \left| u_{\textrm{ext},R}-u_{\mathrm{ext}} \right|_{H^1_*(\Omega\setminus\Gamma)}
    + \frac{1}{\sqrt{\varepsilon}} \left( \Vert \alpha_R - \alpha\Vert_{L^2(\Gamma)}
    +
    \Vert \wideparen{U}_{\mathrm{int},R}- \wideparen{U}_{\mathrm{int}}\Vert_{L^2(\Gamma, BL(\widehat{\Omega}))}
    + \Vert u_{\infty,R} - u_{\infty} \Vert_{L^2(\Gamma)}\right) \\
    \leq  \left(1 + \frac{C}{\gamma}\right) \left( \left| v_{\textrm{ext},R}-u_{\mathrm{ext}} \right|_{H^1_*(\Omega\setminus\Gamma)}
    + \frac{1}{\sqrt{\varepsilon}} \bigg(
    \Vert \beta_R - \alpha\Vert_{L^2(\Gamma)}
    +
    \Vert \wideparen{V}_{\mathrm{int},R}- \wideparen{U}_{\mathrm{int}}\Vert_{L^2(\Gamma, BL(\widehat{\Omega}))}
    + \Vert v_{\infty,R} - u_{\infty} \Vert_{L^2(\Gamma)}
    \bigg)\right).
\end{multline*}
Now, taking $v_{\mathrm{ext},R} = u_{\mathrm{ext}}$, $\beta_R = \alpha$ and $v_{\infty,R} = u_\infty$ we find that
\begin{multline}
   \left| u_{\textrm{ext},R}-u_{\mathrm{ext}} \right|_{H^1_*(\Omega\setminus\Gamma)} \\
    + \frac{1}{\sqrt{\varepsilon}} \left( \Vert \alpha_R - \alpha\Vert_{L^2(\Gamma)}
    +
    \Vert \wideparen{U}_{\mathrm{int},R}- \wideparen{U}_{\mathrm{int}}\Vert_{L^2(\Gamma, BL(\widehat{\Omega}))}
    + \Vert u_{\infty,R} - u_{\infty} \Vert_{L^2(\Gamma)}\right) \\
    \leq  \left(1 + \frac{C}{\gamma}\right) \frac{1}{\sqrt{\varepsilon}} \Vert \wideparen{V}_{\mathrm{int},R}- \wideparen{U}_{\mathrm{int}}\Vert_{L^2(\Gamma, BL(\widehat{\Omega}))}\ .
    \label{eq:coupled_system:truncated:bestappr_err}
\end{multline}
Using the interpolant $\Pi_R$, defined in~\eqref{eq:Pi_R:def}, we can assert in analogy to the proof of Lemma~\ref{lem:UparenR_alpha:err} for a fixed $\vx \in \Gamma$ that
\begin{align*}
   \Vert \wideparen{U}_{\mathrm{int}}(\vx,X,Y)- (\Pi_R \wideparen{U}_{\mathrm{int}}) (\vx,X,Y)\Vert_{BL(\widehat{\Omega})} \leq C |u_\infty(\vx)| \exp(-\pi R)\ .
\end{align*}
Now, taken the $L^2(\Gamma)$-norm on both sides we find that
\begin{align*}
    \Vert \wideparen{U}_{\mathrm{int}} - (\Pi_R \wideparen{U}_{\mathrm{int}} ) \Vert_{L^2(\Gamma, BL(\widehat{\Omega}))} \leq C \|u_\infty \|_{L^2(\Gamma)} \exp(-\pi R)\ .
\end{align*}
Inserting $\wideparen{V}_{\mathrm{int},R} = \Pi_R \wideparen{U}_{\mathrm{int}}$ in~\eqref{eq:coupled_system:truncated:bestappr_err} and as $\|u_\infty \|_{L^2(\Gamma)}$ is bounded by assumption the inequality~\eqref{eq:coupled_system:truncated:err} follows and the proof is complete.
\end{proof}

\end{document}